\documentclass[a4paper,12pt,oneside,reqno]{amsart}

\usepackage{amssymb,amsfonts,amsmath,amsthm,amscd, bbm} %amsaddr
\usepackage[ansinew]{inputenc}%sonderzechen (ansinew windows) WICHTIG
\usepackage[T1]{fontenc}%Umlaute und Akzente WICHTIG
\usepackage{graphicx, color}

\usepackage{mathptmx}%Schriftarten: mathpazo mathptmx helvet avant courier chancery bookman newcent charter

\usepackage{mathdots}
\usepackage[english]{babel} 
\usepackage[a4paper]{geometry} %Dina4-paper
\usepackage{amssymb,amsmath,amsfonts,amsthm} % AMS-Symbols
\usepackage{enumerate} 
\usepackage{pict2e,pstricks,pst-plot,pst-node}
\usepackage{bbm} 
\usepackage{dsfont} 

\newtheorem{theorem}{Theorem}
\newtheorem{lem}[theorem]{Lemma}
\newtheorem{prop}[theorem]{Proposition}
\newtheorem{rem}[theorem]{Remark}
\newtheorem{cor}[theorem]{Corollary}
\newcommand{\R}{\mathbb{R}}

\newcommand{\N}{\mathbb{N}}

\newcommand{\I}{\mathbf{1}}
\newcommand{\PP}{\mathbb{P}}
\newcommand{\E}{{\mathbb{E}}}
\newcommand{\EE}{{\sf{E}}}

\newcommand{\defi}{\equiv}

\newcommand{\be}{\beta}

\newcommand{\bm}{\mathbf{m}}

\newcommand{\by}{\mathbf{Y}}
\newcommand{\md}{  \mathfrak d}

\newcommand{\beq}{\begin{equation}} 
\newcommand{\eeq}{\end{equation}} 
\newcommand{\bea}{\begin{aligned}}
\newcommand{\eea}{\end{aligned}}
\newcommand{\bdm}{\begin{displaymath}}
\newcommand{\edm}{\end{displaymath}}
\newcommand{\barr}{\begin{array}}
\newcommand{\earr}{\end{array}}
\newcommand{\ben}{\begin{enumerate}}
\newcommand{\een}{\end{enumerate}}
\newcommand{\bde}{\begin{description}}
\newcommand{\ede}{\end{description}}

\DeclareMathOperator{\Var}{Var}
\DeclareMathOperator{\Cov}{Cov}

\DeclareMathOperator{\On}{On}

\DeclareMathOperator{\len}{dim}
\DeclareMathOperator{\Tr}{Tr}
\DeclareMathOperator{\diag}{diag}

\begin{document}

\title[AMP Algorithms AND Stein's method]{AMP Algorithms AND Stein's method: \\ Understanding TAP equations with a new method}

\author[S. Gufler]{Stephan Gufler}
\address{S. Gufler  \\ J.W. Goethe-Universit\"at Frankfurt, Germany.}
\email{stephan.gufler@gmx.net}

\author[A. Schertzer]{Adrien Schertzer}
\address{A. Schertzer \\  Institut f\"ur Angewandte Mathematik, Bonn University, Germany }
\email{aschertz@uni-bonn.de}

\author[M. A. Schmidt]{Marius A. Schmidt}
\address{M. A. Schmidt \\ J.W. Goethe-Universit\"at Frankfurt, Germany.}
\email{m.schmidt@mathematik.uni-frankfurt.de}

\thanks{This work was partly funded by the Deutsche Forschungsgemeinschaft (DFG, German Research Foundation) under Germany's Excellence Strategy - GZ 2047/1, Projekt-ID 390685813 and GZ 2151 - Project-ID 390873048, through Project-ID 211504053 - SFB 1060, and by DFG research grant contract number 2337/1-1 and KI 2337/1-2, project 432176920. We are grateful to Nicola Kistler for suggesting the problem and for helpful discussions. We are grateful to Jan Lukas Igelbrink and Jan H\k{a}z\l{}a for interesting discussions.}

\begin{abstract} We propose a new iterative construction of solutions of the classical  TAP equations for the Sherrington-Kirkpatrick model,  i.e.  with finite-size Onsager correction.  The algorithm can be started in an arbitrary point,  and converges up to the AT line.  The analysis relies on a novel treatment of mean field algorithms through Stein's method. As such,  the approach also yields weak convergence of the effective fields at all temperatures towards Gaussians,  and can be applied,  upon proper alterations,  to all models where TAP-like equations and a Stein-operator are available.  \end{abstract}

\maketitle

\section{Introduction and main results}
Let $N\in \N$ and consider independent standard Gaussians ${\bf G} \defi \{g_{ij}\}_{1 \leq i < j \leq N}$ issued on some probability space $(\Omega,  \mathcal F,  \PP)$. We set $g_{ij}= g_{ji}$ for $i>j$. For simplicity of notation, we set $g_{ii}=0$ as well as all the partial derivatives $\frac{\md}{\md g_{ii}}= 0$. We denote by $\E$ expectation with respect to these random variables.  To inverse temperature $\be\in \R$ and external field $h\in \R$,  the TAP equations \cite{TAP} for the SK model \cite{SK} are  self-consistency equations, which approximately describe the quenched magnetizations $\bm = (m_1, \dots, m_N) \in [-1,1]^N, $ i.e. the mean of the Ising spins under Gibbs measure, reading
\begin{equation}\label{e:TAP}
m_i = \tanh\left(h+\frac{\beta}{\sqrt{N}} \sum_{j=1}^N g_{ij} m_j - \beta^2 \left(1-q_N\right) m_i\right) \qquad i=1 \dots N, 
\end{equation}
where 
\beq \label{qn}
q_N \defi q_N(\bm) := \frac{1}{N} \sum_{i=1}^N m_i^2
\eeq
is the {\sf Edwards-Anderson} order parameter,  EA for short,  whereas the "reaction term"
\beq \label{Onsi}
\beta^2 \left(1-q_N\right) m_i =: \On_i
\eeq
is the {\sf finite-size Onsager correction}. These equations are known to be a good approximation for the magnetizations in some high temperature regime \cite{Ad, C, T1, T2} and believed to be such at least if $\be, h$ satisfy the AT condition 
\begin{equation}\label{e:AT}
	\beta^2\EE\frac{1}{\cosh^4(h+\beta\sqrt{q}Z)} \le 1,
\end{equation}
where $q\geq 0$ solves 
\beq \label{e:q}
q= \EE \tanh^2\left( h + \be \sqrt{q} Z\right) =: \phi(q)\,.
\eeq
We refer to $q$ as the {\sf limiting fixed point}, but it is known as the {\sf replica symmetric solution} also; here and henceforth,  we denote by $Z$ a standard Gaussian,  and by $\EE$ its expectation.  
From now on, $q$ denotes the solution of the fixed point equation~\eqref{e:q} which is unique for $h\ne0$ or $h=0, |\beta|\le 1$, and we choose $q$ to be the unique positive fixed point of \eqref{e:q} for $h=0$, $|\beta|>1$.\footnote{ For $h=0$ and $|\beta|>1$ there are two solutions of \eqref{e:q} namely $0$ and $q>0$. See Lemma \ref{step1_lem} for more details. }

For $i=1 \dots N$,  we refer to 
\beq
h_i \defi h_i(\bm) :=  \sum_{j= 1}^N g_{ij}m_j , 
\eeq
as {\sf effective fields}.  In this note, we propose a new algorithmic construction of the solutions of \eqref{e:TAP}. The key observation towards this goal concerns the "true" nature of the finite-size Onsager correction.  Indeed,  by {Gaussian partial integration} (PI for short)  and Leibniz rule, 
\beq \bea \label{gpi}
\E\left[ \frac{\be}{\sqrt{N}} \sum_{j=1}^{N} g_{ij} m_j \right] & \stackrel{(PI)}{=} \frac{\be}{\sqrt{N}} \sum_{j=1}^{N} \E \left[ \frac{d m_j}{d g_{ij}}  \right]  \stackrel{\eqref{e:TAP}}{=}  \E\left[ \be^2 \left( 1 - q_N \right) m_i \right] + R_N \,,
\eea \eeq	
 where
\beq \label{restino}
R_N :=
  \E \left[  \frac {\be^2}{\sqrt{N}} \sum_{j\neq i} \left( 1-m_j^2\right) \left\{ \sum_{l\neq j} \frac{g_{lj}}{\sqrt{N}} \frac{d m_l}{d g_{ij}} - \be  \frac{d}{d g_{ij}}  \left\{  \left( 1-q_N \right) m_j \right\}  \right\}-\frac{\be^2}{N} \left( 1 - m_i^2 \right) m_i\right] \,.
\eeq 

The first term on the r.h.s. of \eqref{gpi} is the (mean of the) Onsager correction \eqref{Onsi}: it neatly emerges through one Gaussian PI only.  The second term captures the nifty fact that the $\bf G$-dependence of the
 $m_i = m_i(g_{kl},  1 \leq k < l \leq N)$ actually goes on  {\it ad infinitum}.  Notwithstanding,  a closer look at \eqref{restino} suggests that stochastic cancellations may dampen such infinite rebouncing. In other words,  we expect that $R_N$ contributes  to lower orders only.  This insight turns out to be correct,  cfr. Theorem \ref{derivcontrolthrm} below,  and represents the backbone of the present work.  Due to the Gaussian nature of the random sum,  a first natural guess is of course a central limit theorem;  this will indeed turn out to be the case,  see Proposition \ref{weak-conv-clt} below. To elaborate precisely we define our iteration as follows. For $i\leq N$ and $i,N,k \in \N$ we consider
\begin{equation}\label{TAP-rec}
m^{(k+1)}_i := \tanh\left( h+Y_i^{(k)}\right)
\end{equation}
with the re-centered effective fields given by 
\beq
 Y^{(k)}_i  := \frac{\beta}{\sqrt{N}} \sum_{j=1}^N g_{ij} m^{(k)}_j - \On^{(k)}_i,  
\eeq
where $\On^{(k)}_i $ is given by either of the following choices 
\begin{enumerate}[(I)]
\item \label{On-cla} $\qquad\qquad\qquad \qquad \On^{(k)}_i := \beta^2 \left(1-  q^{(k)}_N \right)  m^{(k-1)}_i,$ \\
\item \label{On-new} $\qquad\qquad\qquad\qquad \On^{(k)}_i := \frac{\beta}{\sqrt{N}} \sum_{j=1}^N {\mathfrak d}_{ij} m_j^{(k)}\,.$ \\
\end{enumerate}
Here and henceforth we use the notations 
\[q^{(k)}_N:=\frac{1}{N} \sum_{i=1}^N {m^{(k)}_i}^2 \quad \mbox{ and} \quad {\mathfrak d}_{kl} F := \frac{d F}{d g_{kl}},\]
for suficiently smooth functions $F$.  Choice \eqref{On-cla} is the {\it classical},  finite-size Onsager correction to the TAP equations \eqref{e:TAP}. For this choice the iteration is known to approximate the magnatizations of the SK-Model well, at least in some temperature regime, see \cite{CT} for details. By the heuristics related to \eqref{gpi},  the second choice \eqref{On-new} rather highlights the Onsager reaction term as a re-centering of the effective fields.  All our results are valid in both cases,  yet the second (and apparently: new) correction -- as well as its interpretation -- are arguably more natural within our Stein framework.  
To properly define $\bm^{(k)}$ a starting values $\bm^{(1)}$ or in case $\eqref{On-cla}$ $\bm^{(0)},\bm^{(1)}$ are needed. Concening these values we assume throughout, that :
\begin{itemize}
\item[A1)] The starting value $\bm^{(1)}\in[-1,1]^N$ is either deterministic or independent of the $\bf G$-disorder.
\item[A2)] In case of choice \eqref{On-cla},  we additionally set $\bm^{(0)}:=0$.  
\end{itemize}
The extent to which A1) can be relaxed is discussed in Remark~\ref{r:g} below. 

Clearly the limit of this algorithm, if it converges, is a solution of \eqref{e:TAP}. This motivates our first result, the following {\it pseudo-convergence}\footnote{This is quite a delicate issue.  Indeed,  even the meaning of "solution" and "convergence" are  {\it a priori} by far not obvious: due to the mean field character of the random system,  one has to first send $N \to \infty$,  and only in a second step $k \to \infty$.  Naturally,  the proper procedure to allow for a discussion of the true TAP-solutions would send $k \to \infty$ with $N$ fixed. } :

\begin{prop}\label{c:convergence}
For $\beta,h \in \R$ and $\liminf_{N\rightarrow\infty} q^{(1)}_N>0$ if $h=0,|\beta|>1$,
 we have 
\begin{equation}\label{convalgo} 
	\lim_{k, l \to \infty} \limsup_{N \to \infty} \frac{1}{N} \sum_{i=1}^N \left[ \left( m_i^{(k)} - m_i^{(l)} \right)^2 \right] = 0
	\end{equation}
if and only if the AT condition~\eqref{e:AT} is satisfied. 
\end{prop}
To investigate further we shorten 
\begin{equation}\label{def-qkk}
	q^{(k,k')}_N :=  N^{-1} \sum_{i}^{N} m^{(k)}_i m^{(k')}_i  \quad \text{and} \quad q^{(k)}_N := q^{(k,k)}_N\,.
\end{equation}
 The convergence in \eqref{convalgo} goes as follow: We can rewrite the l.h.s. of \eqref{convalgo} as 
\beq \label{weird2}
\frac{1}{N} \sum_{i=1}^N \left[ \left( m_i^{(k)} - m_{i}^{(l)} \right)^2\right]=  q^{(k)}_N+ q^{(l)}_N-2  q^{(k,l)}_N \, .
\eeq
We will see that the two first terms of the r.h.s. tend to $q$ as $N \to \infty$. However, $q^{(k,l)}_N$ tends to $q$ as $N \to \infty$ and $k \to \infty$ if and only if the AT-condition  \eqref{e:AT} is satisfied and converges to a limit strictly smaller than $q$ if the AT-condition \eqref{e:AT} is not satisfied. \\
To describe the situation precisely let 
\begin{equation}\label{e:psi}
	\psi(t,t',t''):=\EE\left[\tanh(h+\beta Z')\tanh(h+\beta Z'') \right]
\end{equation}
for $t',t''\in [0,1]$, $|t|\leq \sqrt{t' t''}$, where the expectation is over centered Gaussians $Z',Z''$ such that 
\beq 
\Var(Z')=t',  \quad \Var(Z'')=t'', \quad \Cov(Z',Z'')=t.
\eeq
Additionally we say $X_{N,k}\rightarrow X$ in probability as $N\to\infty$  followed by $ k\to\infty$, if for all $\varepsilon>0$
$$ \lim_{k\rightarrow\infty} \limsup_{N\rightarrow \infty}\PP\left( |X_{N,k}-X|>\varepsilon \right)   = 0.$$
In the same way, we say $X_{N,k}\rightarrow X$ weakly as $N\to\infty$  followed by $ k\to\infty$, if 
$$ \lim_{k\rightarrow\infty} \limsup_{N\rightarrow \infty} \E \left( f\left(X_{N,k}\right) \right)= \E \left( f\left(X\right) \right)$$
for all continuous and bounded functions  $f$. Similarly, for $k_1,k_2 \dots k_L \rightarrow \infty$ the $\lim_{k\rightarrow\infty}$  is repaced by $\lim_{k_1,k_2 \dots k_L\rightarrow\infty}$.
\begin{rem}
The use of $\limsup$ is required to account for arbitrary starting values. It can be replaced by the normal limit if the initial value is well behaved, i.e if $q_N^{(1)}$ and $\frac{1}{N}\sum_{i=1}^{N}m_i^{(k)}$ converge as $N\rightarrow \infty$.
\end{rem}

\begin{theorem}\label{thm-contract-cond}
	
		 For all $\beta,h \in \R$ and assuming $\liminf_{N\to\infty} q^{(N)}_1>0$ a.s. additionally if $h=0, |\beta|>1$ , we have		
		%\beq\label{contract-cond-uni}
		  %q^{(k,k')}_N \to \left\{ { q \text{ if } k=k' \atop \tilde{q} \text{ if } k\neq k' } \right.  \text{ in probability as } N\to\infty \text{ followed by } k,k'\to\infty
		%\eeq
		
		\beq\label{contract-cond-uni}
q^{(k,k')}_N \to \left\{
    \begin{array}{ll}
        q  & \mbox{if }  k=k' \\ 
        \tilde{q} & \mbox{if } k\neq k'  
    \end{array}
\right.
\text{ in probability as } N\to\infty \text{ followed by } k,k'\to\infty
\eeq

where $\tilde{q}$ is the smallest non-negative fixed point of $ f(t) = \psi(t,q,q)$. Additionally $q=\tilde{q}$ if the AT condition~\eqref{e:AT} is satisfied and $q>\tilde{q}$ otherwise.  

\end{theorem} 	

\begin{rem}
For $h=0$, the system can not entirely forget the initial condition in the double limit in low temperature. Indeed,  if $h=0$, $|\beta|> 1$ and $\lim_{N\to\infty} q^{(1)}_N =0$, then $q>0$, but $ q^{(k,k')}_N \to 0.  \text{ in probability as } N\to\infty$. This case is due to the fact that $t \to \psi(t,t,t)=\phi(t)$ has two fixed points for $h=0$, $|\beta|> 1$ (see equation~\eqref{e:q} and the comments below): A repulsive fixed point $0$ and an attractive fixed point $q$. Therefore, $0$ is only relevant, when the algorithm starts "close" to it.
\end{rem}

\begin{proof}[Proof of Proposition \ref{c:convergence}]
The claim follows immeadiately by applying Theorem \ref{thm-contract-cond} to \eqref{weird2}.
\end{proof}
We now state the afforementioned CLT for the centered effective fields, which concludes a first big picture overview over the behaviour of the iteration.
\begin{prop} \label{weak-conv-clt}
 For all $\beta,h \in \R$, $L\in\N$, $(k_l,i_l)_{l\leq L}\in (\N \times \{1,\cdots,N\})^L$, and if \\ \noindent$h=0, |\beta|>1$ assuming $\liminf_{N\to\infty} q^{(N)}_1>0$ a.s. additionally, we have

\[(Y_{i_1}^{(k_1)},\ldots, Y_{i_L}^{(k_L)}) \rightarrow {\bf \mathcal Z},\] 
in distribution as $N\to\infty$ followed by $k_1,\ldots,k_L\to\infty$, such that $(i_l,k_l)\neq (i_{l'},k_{l'})$ for all $l\neq l'$ and where ${\bf \mathcal Z}$ is a centered Gaussian field with covariance  
\[\Sigma_\infty(l,l'):=\EE[{\bf \mathcal Z}_l {\bf \mathcal Z}_{l'}] = \begin{cases} \beta^2 q & \mbox{ if } l=l' \\ \beta^2 \tilde{q} & \mbox{ if } l\neq l', i_l=i_{l'} \\ 0 & \mbox{ otherwise } \end{cases}.\]
\end{prop}
Note that this implies that $(Y_{1}^{(k_1)}, Y_{2}^{(k_2)}, \dots Y_L^{(k_L)})$ tends weakly to a vector of independent centered Gaussians of variance $\beta^2 q$. Correlations are only present in the limit between effective fields at the same site, i.e. between $Y_i^{(k)}$ with the same $i$. If the AT condition \eqref{e:AT} is satisied $q=\tilde{q}$ by Theorem \ref{thm-contract-cond}, which corresponds to the re-centered effective fields being perfetly correlated and therefore aligning in the double limit. If however the AT condition is  not satisfied, we observe $\tilde{q}<q$ , hence in the double limit the centred effective fields are not perfectly correlated. They share a macroscopic Gaussian as a common trunk, but retain an also macroscopic independent Gaussian part, that is reshuffled each iteration. 

We will see through Stein's method that the weak convergence of the re-centred effective fields give a new perspective on the AMP algorithms: In fact, the observed weak convergence and law of large numbers are two sides of the same coin. To explain this further we proceed with a high level sketch in the next section. Thereafter we discuss and compare to previus work on the topic in Section \ref{bolth-sec}. We then proceed to proof all results step by step: discussing the approximate Gaussianity of effective fields by Stein's method in section \ref{steinTAP-sec} and a law of large numbers for covariances in Section \ref{covlocf-sec}. This leads to a deterministic propagation of variances and covariances in the $N$-limit, which requires a fixed point and contractivity analysis of $\phi$ and $t\rightarrow \psi(t,q,q)$ discussed in Section \ref{fp-sec} to conclude proofs in Section \ref{contract-sec}.

\section{Sketch of the proof and the image}

 In order to prove convergence of the algorithm \eqref{TAP-rec} we develop a treatment which radically differs from \cite{b1, b2}: in a first step we will establish, by Stein's method,  optimal $N-$bounds on the distance between the law of the random vector $(Y_i^{(k)})_{(i,k)\in I}$ for some fixed finite $I\subset \N\times\N$ and a "conditional" multivariate Gaussian, cfr. Theorem \ref{Stein-TAP} below.  The Gaussian vector is centered with entries of variance 
$\be^2 q^{(k)}_N$, which are independent for different sites (values of $i$).  

By exchangeability,  we will then deduce that 
\beq \label{e:psihatdef}
q_N^{(k+1)} \defi \frac{1}{N} \sum_{i=1}^N  \tanh\left( h+ Y_i^{(k)} \right)^2 \approx \EE\left[ \tanh\left( h+ \be \sqrt{q_N^{(k)}} Z \right)^2 \Bigg{|} q_N^{(k)} \right] \defi \phi\left( q_N^{(k)}\right)\,.
\eeq

In a third and final step we show contractivity of the $ \phi$-function for any $\beta,h$, see Theorem \ref{unif_conv_thrm}. The convergence of the algorithm will then follow from the convergence of $q_N^{(k,k')} \to q$ when $N$ and then $k,k'$ $\to \infty$ under the AT condition \eqref{e:AT}. 

\noindent For $L \in \N$ and any choice of sites and iteration numbers $i_1, \dots, i_L, k_1, \dots, k_L \in \N$ and \\ \noindent$N\geq \max_{l\leq L} i_l$,  consider the random vector 
\beq
\by := (Y_{i_l}^{(k_l)})_{l\leq L}\,.
\eeq

For $r \in \N$ we define the following spaces of test functions 
\begin{equation}  
\mathcal{H}^r := \{ f\in \mathcal{C}^{r+1}(\R^L,\R): |\partial_\alpha f|_\infty \leq 1 \mbox{  for } |\alpha|\in\{1,\cdots, r\}, |\partial_\alpha f|_\infty<\infty \mbox{ for } |\alpha|=r+1 \},
\end{equation} 
where $\alpha$ goes over every $|\alpha|$-fold partial derivative (see \eqref{multiindex1}, \eqref{multiindex2} for details). With this we define the associated distances between the laws of integrable random variables on $\R^L$ 
\begin{equation} \label{e:dr}
d_{r}(X,X') = \sup_{f\in \mathcal{H}^r}|\E[f(X)]-\E[f(X')]| \, .
\end{equation}
 \emph{Integral probability metrics} of this kind are well known and have been studied in e.g.~\cite{GL} and the references therein. We also consider a random vector 
\beq
{\bf  Z} = ( Z_{i_l}^{(k_l)})_{l \leq L}\,,
\eeq 
which,  {\it conditionally} on the $q_N^{\cdot, \cdot}$-s,  is a multivariate centered Gaussian with covariance 
\beq \label{covarianza}
\Sigma(l, l') := \EE\left[   Z_{i_l}^{(k_l)}  Z_{i_{l'}}^{(k_{l'})}  \right] = 
\begin{cases}
\beta^2 q^{(k_l,k_{l'})}_N  & \text{if} \; i_l = i_{l'}, \\
0 & \text{otherwise},
\end{cases} \qquad 
l, l' = 1 \dots L \,.
\eeq

\begin{theorem}(Proximity of the effective fields to a Gaussian).  \label{Stein-TAP} For all $\beta,  h \in \R$ and any $L\in\N$, $(k_l,i_l)_{l\leq L}\in (\N \times \{1,\cdots,N\})^L$, there exists $C\defi C(k, L, \beta, h)>0$  such that   
\begin{equation}
d_{2}\left( \by,  {\bf Z}\right) \leq \frac{C}{\sqrt{N}}\,, \quad \forall {N \in \N}\,.
\end{equation} 
\end{theorem}
The above statement is in itself not conclusive: We still have to prove that the $q$-s,  the scalar products of the iterates , see \eqref{def-qkk}, appearing in the covariance \eqref{covarianza} converge in the large-$N,k$ limit.  Theorem \ref{q-prop} below is the key step towards this fact. We will see that we have an inductive relation between the $q_N^{\cdot, \cdot}$-s: Intuitively, it is due to the fact that
\begin{multline}
q^{(k+1,k'+1)}_N \approx \E q^{(k+1,k'+1)}_N=\sum_{i=1}^N\frac{1}{N}\E\left[ m^{(k+1)}_i m^{(k'+1)}_i\right]=\E\left[\tanh\left(h+Y^{(k)}_1\right)\tanh\left( h+ Y^{(k')}_1\right)\right] \\
\approx \E\left[\tanh\left(h+ Z_1^{(k)}\right) \tanh\left(h+ Z_1^{(k')}\right)\right]   =\E \psi\left( q_N^{(k, k')}, q_N^{(k)}, q_N^{(k')}\right)\approx \psi\left( q_N^{(k, k')}, q_N^{(k)}, q_N^{(k')}\right),\
\end{multline}
the first and last equality by concentration and the first line to the second line by Theorem \ref{Stein-TAP}. The next theorem makes this precise. To state it we define for $t\geq 0$
\begin{equation}
\label{e:chi}\chi(t) := \EE\left[\tanh(h+ \beta \sqrt{t} Z )\right]\,.
\end{equation}
\begin{theorem}\label{q-prop}(Propagation of the covariance entries). For \emph{any} $\beta,  h\in\R$,
and $k,k'\geq 1$ there exists $C\defi C(k,k',\beta, h)>0$ such that:
\begin{equation}
	\label{e:qprop-psi}
	\E\left[\left(q^{(k+1,k'+1)}_N-  \psi\left( q_N^{(k, k')}, q_N^{(k)}, q_N^{(k')} \right) \right)^2\right] \le \frac{C}{\sqrt N}
	\end{equation}
	and
	\begin{equation}
		\label{e:qprop-chi}
		\E\left[\left( q^{(k+1,1)}_N - \chi\left(q^{(k)}_N\right) \frac{1}{N}\sum_{i=1}^N  m^{(1)}_i\right)^2\right]  \le \frac{C}{\sqrt N}.
		\end{equation}
\end{theorem}

  In case that $q^{(1)}_N$ and $\frac{1}{N}\sum_{i=1}^N  m^{(1)}_i$ converge as $N\to\infty$, recursive application of Theorem~\ref{q-prop} implies that also $q^{(k,k')}_N$ has a limit in probability as $N\to\infty$. By analysing the $k,k'\to\infty$ limit, we obtain the convergence of $q^{(k,k')}_N$ from Theorem \ref{thm-contract-cond}, which in turn seamlessly yields the convergence of the algorithm: see Proposition \ref{c:convergence}. \\

What thus stands behind our results, and is fully captured by the Stein's method, is the following beautiful picture: Provided the mean and radius of the starting magnetisations converge we have for any choice of parameters (inverse of temperature and external field) that any fixed family of effective fields converges weakly, in the infinite volume limit ($N \to \infty$), to a Gaussian field with covariances depending deterministically on limiting starting radius and mean. This is due to the following interplay: the Gaussian convergence (see Theorem \ref{Stein-TAP}) of re-centered effective fields implies a law of large numbers (see Theorem \ref{q-prop}) for the limiting covariance structure of re-centered effective fields one iteration further. This procedure explains, how the highly non linear TAP-transformation leaves finitely many effective fields Gaussian and propagates the starting values only through evolution of the covariance: perhaps surprisingly, this is true in any $(\beta,h)$-phase. 

Even more interesting is the limiting behaviour when the second limit is taken, that of the iteration parameters: Here the picture does become (highly) parameter dependent as well as fully independent of the starting conditions. This behaviour emerges as iterating the propagation of the covariance structure approaches an attractive fixed point, which turns out to be unique in all cases. This explains the vanishing dependence on starting conditions\footnote{except for $h=0$ and $m^{(1)}\approx 0$ }.  See Theorem \ref{thm-contract-cond} and Proposition \ref{weak-conv-clt}. The temperature dependence crucially comes into play when analysing the fixed point. The system stabilises around a covariance $\tilde{q}$. If the AT condition is satisfied we observe $\tilde{q}=q$, which means that the effective fields from iteration to iteration macroscopically align, i.e. we obtain approximate solutions of the TAP equation. If the AT condition is not satisfied we observe $\tilde{q}<q$ hence the re-centered effective fields, partially align around a singular Gaussian vector shared among all iterations, but they retain some macroscopic randomness, which is independently reshuffled each iteration.  See Theorem \ref{thm-contract-cond}. We now remark on the special case of vanishing external field which exhibits some special behaviour:
\begin{itemize}
	\item
for vanishing external field and  under the AT condition, i.e. $h=0$ and $|\beta| \leq 1$, the iterates converge towards a degenerate Gaussian field of vanishing variance; this is due to the variance propagation $ \phi$ having $0$ as a fixed point for $h=0$, which is unique and attractive iff $|\beta|\leq 1$:  indeed the variances of the effective fields turn out to be approximate solutions of  the limiting fixed point equation for the order parameter of the theory (the $q$-function, in this case a scalar), and indeed converge upon iteration to the unique fixed point. As in this case the unique solution is zero; everything collapses to this trivial limit.

\item for vanishing external field and $|\beta| > 1$ (in low temperature) the variances of the limiting Gaussian field still converge to the same fixed point, and so do the covariances (!). Yet, and crucially: the limiting variance now is positive\footnote{ except for $m^{(1)}\approx 0$}, while the covariances remain zero. Hence the effective fields are fully reshuffled each iteration, leaving the iteration chaotic permanently. 

\end{itemize}
Our proof via Stein thus unravels multiple aspects of the intricate behaviour of mean field spin glasses, in particular those pertaining to the delicate interplay of almost sure phenomena (such is the convergence of the elements of the covariance matrix of the limiting Gaussian fields) and weak limits (of the effective fields). As a matter of fact, it is impossible to say which of these phenomena comes first : things are deeply intertwined, and can hardly be disentangled. 

As a final remark we want to mention, that this new method does not really depend on the Gaussian nature of the disorder and the implementation of this method for other mean fields models is left to future research. 

\section{Previous work: The Bolthausen's Algorithmus}\label{bolth-sec}

\noindent An algorithmic construction of the solutions to a {\it modified} version of \eqref{e:TAP} has been obtained by Bolthausen \cite{b1, b2} (see also \cite{CT} for another iterative scheme),  who studied the {\it approximate version} of TAP equations
\beq \label{mean4}
m_i = \tanh\left(h + \frac{\be}{\sqrt{N}} \sum_{j=1}^N g_{ij} m_j  - \be^2 \left( 1-q \right) m_i\right)\qquad i=1 \dots N, 
\eeq
where the EA-parameter $q_N$ from \eqref{e:TAP} is replaced by a scalar\footnote{Corresponding,  albeit only {\it a posteriori},  to the $\PP$-almost sure $q_N$-limit.} $q$.
To avoid confusion,  we shall refer to the {\it modified} TAP equations \eqref{mean4} as {\sf B-TAP}.  Bolthausen's algorithm for {\sf B-TAP} is initialised with the deterministic values 
\beq \label{init}
\bm^{(0)}_i := {\boldsymbol 0},  \qquad \bm^{(1)}_i := \sqrt{{\boldsymbol q}}, \quad  \forall i  \leq N
\eeq
which are then propagated by a two step Banach procedure (${k-1, k} \to k+1$) given by
\beq \bea \label{e:two-step}
m_i^{(k+1)} =\tanh\left(h + \frac{\be}{\sqrt{N}} \sum_{j=1}^{N} g_{ij} m_j^{(k)}  - \be^2 \left( 1-q \right) m_i^{(k-1)} \right)\,, \qquad k \geq 1. 
\eea \eeq
For $h\neq 0$, it is shown in the seminal paper \cite{b1} via some delicate analysis involving Gaussian projections,  conditioning techniques,  and law of large numbers that the following {\it pseudo-convergence} holds:
\beq \label{weird}
\lim_{k, l \to \infty}\limsup_{N\to \infty}  \frac{1}{N} \sum_{i=1}^N\E\left[ \left( m_i^{(k)} - m_{i}^{(l)} \right)^2\right]  = 0,
\eeq
if and only if $(\be, h)$ satisfy the AT condition \eqref{e:AT}. \\

\noindent The above treatment,  to our knowledge, is the only procedure available so far for a rigorous analysis of {\sf B-TAP}  up to the AT line, but suffers from a number of shortcomings.

\begin{itemize}
\item First of all,  the Onsager correction: by replacing $q_N$,  the random and finite-size EA-order parameter,  with $q$, the deterministic solution to \eqref{e:q} one artificially injects a posteriori information into the model a priori: It does not allow to understand whether the expected behaviour emerges by itself or was built in by this choice of $q$ and starting point. Indeed,  by work of Plefka \cite{plefka},  the TAP equations are known to be the critical points of the expansion of the Gibbs potential to second order, but the Onsager correction \eqref{mean4} is not compatible with Plefka's framework nor with the original TAP derivation \cite{TAP},  see also \cite{GSS,kistler} for more on this (and related) issue(s).   

\item A second delicate point is the initialisation \eqref{init}: the {\sf B-TAP}-algorithm is known to converge  for this very specific choice of starting value only,  somewhat corresponding to the odd situation where one solves a fixed point equation by starting virtually 'in' the solution.  Since the procedure is but a sophisticated Banach-iteration,  this restriction is unnatural,  and hides key phenomena concerning the onset of the spin glass phase.  

\item A third point is that the Bolthausen algorithm does not allow to investigate the case $h=0$.

\item Finally we want to make the point, that the sophisticated conditioning technique developed by Bolthausen seems to be dependent on the Gaussian nature of the disorder, but certainly can only work if at least the effective fields are approximately Gaussian. The method presented here is more robust in the sense that approximate integration by parts and Stein operators for different distributions are established tools, which allow approaching these scenarios.

\end{itemize}

\noindent  Nevertheless these magnetizations, provided by the pseudo-convergence, allowed Bolthausen \cite{b2} to give a proof of the replica symmetric formula for the free energy of the SK model in high temperature. Brennecke and Yau \cite{br} proved an extension of this result. However, neither result is valid for all $(\be, h)$ satisfying the AT condition \eqref{e:AT}, which is believed to separate the high-temperature from the low-temperature  regime. Extending aforementioned results up to AT is an open problem.\\

\noindent  Bolthausen's algorithm and conditioning techniques have found wide applications,  in particular they have been used by Montanari and coauthors to analyse Approximate Message Passing (AMP for short) in a large class of models,  see \cite{AMS,  M1, M2} and references therein,  see also \cite{ZK} for a theoretical physics' take.  As our approach does not require Gaussian disorder we belief that this new method is applicable to an even wider range of models and problems. Also considering the possibility of adapting Stein's method e.g. to the diluted case where the limiting distribution of effective fields is expected to be non Gaussian, by considering an appropriate Stein operator.  \\

\section{Stein TAP}\label{steinTAP-sec}
\subsection{Overview}
The main objective of this section is to prove Theorem \ref{Stein-TAP} using Stein equation. Before we tackle this task we prove Proposition \ref{weak-conv-clt} from Theorem \ref{thm-contract-cond} and \ref{Stein-TAP}.
\begin{proof}[Proof of Proposition \ref{weak-conv-clt}]
	  Clearly 
	\begin{equation}\label{e:d2-cor}
		d_{2}\left( \by,  {\bf \mathcal Z}\right) \leq d_{2}\left( \by,  {\bf  Z}\right)+d_{2}\left( {\bf Z},  {\bf \mathcal Z}\right),
	\end{equation}
where the first term tends to $0$ as $N \to \infty$ by Theorem~\ref{Stein-TAP}. By Theorem~\ref{thm-contract-cond} $\Sigma$ converges to $\Sigma_\infty$ as $N\to\infty$ followed by $k_1,..,k_L\to\infty$. Now we consider the simple coupling ${\bf Z} = \Sigma^{1/2} Z$ and ${\bf \mathcal Z }= \Sigma_\infty^{1/2} Z$, for a standard Gaussian vector $Z$ independent of $\Sigma$ and use for $f\in \mathcal{H}^2$ the estimate 
\begin{equation} \label{simple_estimate}
\E\EE|f({\bf \mathcal Z})- f({\bf Z})| \leq |\|\nabla f\|_2|_\infty \;\EE \|Z\|_2 \E\| \Sigma^{1/2}-\Sigma^{1/2}_\infty\|_{op} \leq L \E\| \Sigma^{1/2}-\Sigma^{1/2}_\infty\|_{op},
\end{equation}
Since all operations involved on the r.h.s. are continuous in the entries of $\Sigma$ we see that the second term of \eqref{e:d2-cor} vanishes with $N\to \infty$ followed by $k_1,..,k_l \rightarrow \infty$, since the operator norm involved is bounded: 
$$\| \Sigma^{1/2}\|_{op} \leq \| \Sigma^{1/2}\|_{F}\leq L |\Sigma^{1/2}|_\infty \leq L \sqrt{|\diag(\Sigma)|_\infty} \leq L \beta,$$
where $\|.\|_F$ is the Frobenius norm. From \eqref{e:d2-cor} vanishing, we obtain the asserted weak convergence, since the vector space of finite linear combinations of $\mathcal{H}^2$ functions is dense in the continuous functions when restricting to a compact by Stone Weierstrass Theorem. The restriction to compact support is then removed using tightness of both sequences, which holds since the covariance matrices in question have entries bounded uniformly in $N$.   
\end{proof}

We will work with test functions in $\mathcal H^2$, since it is suitable to prove  weak convergence,  includes $\tanh$ and matches nicely with the regularity the Stein operator provides. To accomodate for $\Sigma$  in Theorem \ref{Stein-TAP} being random we slightly modify the classical Stein equation.	One easily checks that $\Sigma$ is in fact symmetric and positive semi-definite, and consequently $\Sigma^{1/2}$ exists and is symmetric. For $Z$ a standard Gaussian on $\R^L$, $F\in \mathcal{C}^3(\R^L,\R)$ with $|\partial_\alpha F|_\infty <\infty$ for $|\alpha|\in\{1,2,3\}$ we define 
	$$V_F(x) = -\int_0^\infty \EE_Z\left[ F(x e^{-u}+ \sqrt{1-e^{-2u}}Z) \right]-\EE_Z[F(Z)]  du.$$
	By the multivariate Stein equation  (see e.g.\ Lemma 2.6 in~\cite{CGS}),
	\begin{equation} \label{SteinOP} F\left(x \right) - \EE_Z F\left(Z\right)	= \Tr \nabla^2V_F\left(x\right)- x^T \nabla V_F\left(x\right).\end{equation}
	For $f\in \mathcal H^2$ we define  
	$$ v(y) = -\int_0^\infty \EE_Z\left[ f( y e^{-u}+ \sqrt{1-e^{-2u}} \Sigma^{1/2}Z) \right]-\EE_Z[f(\Sigma^{1/2}Z)]  du,$$
	where we recall that $\EE_Z$ is the expectation only over $Z$, which is chosen to be independent of everything else, i.e. of the $\left(g_{ij}\right)$ and $m^{(1)}$, hence also of  $\Sigma$.
	
	Observe that by choosing 
	$$ F\left( y\right)=f\left(\Sigma^{1/2}y\right), \quad \mbox{ verifying } \quad V_F\left(y\right)= v(\Sigma^{1/2}y)$$
	and plugging into \eqref{SteinOP} with $x= \Sigma^{1/2}y$, we obtain 
	\begin{equation}\label{SteinOP2}  f\left(x \right) - \EE_Z f\left(\Sigma^{1/2}Z\right)	= \Tr \left(\Sigma\nabla^2  v \left(x\right)\right)- x^T \nabla v\left(x\right).\end{equation}
	Now note that for $|\alpha|\in\{1,2,3\}$
	\begin{equation}\label{e:gl}
		\partial_\alpha v(y) = -\int_0^\infty e^{-u|\alpha|} \EE_Z\left[ \partial_\alpha f( y e^{-u}+ \sqrt{1-e^{-2u}} \Sigma^{1/2}Z) \right]du,
	\end{equation}
	which gives us $v\in \mathcal H^2$, since $f\in \mathcal H^2$ by assumption and 
	\begin{equation}\label{steinest}|\partial_\alpha v|_\infty \leq \frac{1}{|\alpha|} |\partial_\alpha f|_\infty. \end{equation}
	Using $x= \left(Y^{(k_l)}_{i_l}\right)_{l\leq L}$ in \eqref{SteinOP2}, writing $v_l= \partial_{e_l}v$ and taking expectations we have  
	
	\begin{align}\label{e:Stein-tr}
		& \hspace{.5cm} \E f\left((Y_{i_j}^{(k_j)})_{j\leq L}\right) - \E f\left(Z_{\Sigma_{(i,j)}}\right)\\ 
		&=  \E\left[ \Tr\left(\Sigma \nabla^2  v \left((Y_{i_j}^{(k_j)})_{j\leq L}\right)\right)\right]
		- \E\left[ (Y_{i_j}^{(k_j)})_{j\leq L}^T \nabla v  \left((Y_{i_j}^{(k_j)})_{j\leq L}\right)\right]\notag \\
		&=  \sum_{l=1}^L \E\left[ \beta^2 \sum_{l'=1}^L\delta_{i_l,i_{l'}}\frac{ (m^{(k_l)})^T m^{(k_{l'})} }{N} \partial_{e_{l'}}  v_l\right] -\E\left[  Y_{i_l}^{(k_l)}  v_l  \right] \notag \\
	\end{align}
	where we simplify notation by suppressing the argument of $v_l$, which is $(Y_{i_j}^{(k_j)})_{j\leq L}$. In order to prove Theorem  \ref{Stein-TAP}, it remains to show that the r.h.s. is bounded in absolute value by $C/\sqrt{N}$. To this end we need to control the following quantities: \begin{equation}\label{e:def-DeltaE}
	\Delta^{(k)}_{x;y,z} := \frac{\beta}{\sqrt{N}} \sum_l g_{x,l}  {\mathfrak d}_{yz}\,m_l^{(k)}
	- {\mathfrak d}_{yz}\,  \On_x^{(k)}, \hspace{1cm} 
	\mathcal{E}_{x}^{(k)} := \frac{\beta}{\sqrt{N}}\sum_{l\neq x}  {\mathfrak d}_{xl}\, m_l^{(k)} - \On_x^{(k)}.
\end{equation}
This is the object of study of our next Theorem, which is shown in Section \ref{detasec}. 
\begin{theorem}\label{deltathrm} For $\beta\in \R$, $k\in\N$ there exists $C=C_{k,\beta}>0$, s.t.  
	$$ \E\left[ \left(\Delta^{(k)}_{x;y,z}\right)^2\right] \leq \frac{C}{N^2} \quad \mbox{ and } \quad \E\left[ \left(\mathcal{E}^{(k)}_{x}\right)^2\right] \leq \frac{C}{N}$$
	for all $h\in \R$, $N\in \N$ and $x, y, z\leq N$ with $y\neq z$. 
\end{theorem}
To derive Theorem \ref{deltathrm}, a fundamental understanding of the magnitude of derivatives of $m$ with respect to $g$ is required. To this end, we introduce some notation. First we use multiindex notation, i.e. for 
\begin{equation}\label{multiindex1}  \alpha \in \N_0^{\mathfrak p},\quad \mbox{ where } \quad  \mathfrak p \defi \mathfrak p_N:=\{\{i,j\}:1\le i<j\le N\}\end{equation}
we use the shorthands
\begin{equation}\label{multiindex2}
	|\alpha| := \sum_ {x\in \mathfrak p}\alpha_x, \quad \alpha! := \prod_{x\in \mathfrak p} \alpha_{x}! \quad \mbox{ and } \quad \partial_\alpha  := \prod_{x\in\mathfrak p} (  {\mathfrak d}_{x})^{\alpha_x}\,,
\end{equation}
also recalling the convention ${\mathfrak d}_{ij} = \frac{d}{d g_{ij}}$.   Additionally, for any $e\in\mathfrak p$ we write $\delta_e\in \N_0^{\mathfrak p}$ to refer to the vector with a $1$ in the $e$-coordinate and otherwise zeros. 
Finally, we associate to each multiindex $\alpha \in \N_0^{\mathfrak p}$ the undirected multigraph on $V= \{i: \exists j: \alpha_{\{i,j\}}>0\}$ with $\alpha_{\{i,j\}}$ edges between vertices $i$ and $j$. This allows us to use $\alpha$ as a multiindex and its associated multigraph interchangibly. Consequently $i\in \alpha$ means $\exists j: \alpha_{\{i,j\}}>0$ and $\{i,j\}\in \alpha$ means $\alpha_{\{i,j\}}>0$. 
The following key observation is proved in Section~\ref{derivcontrolsec}.
\begin{theorem} \label{derivcontrolthrm} For $\beta\in \R$, $t,k \in \N$ there exists $C \defi C_{|\alpha|,t,k,\beta}>0$,  such that 
	\begin{equation}\label{claim} \E\left[ \left(\sqrt{N}^{|\alpha|+\mathcal{C}_\alpha(p)} \partial_\alpha m_p^{(k)}\right)^{2t}\right]\leq C\end{equation} 
	for all $\alpha\in \N_0^{\mathfrak p}$, $N\in \N$, $p\leq N$, $h\in \R$ and where $\mathcal{C}_\alpha(p)$ denotes the number of  connected components in $\alpha$ that do not contain $p$.
\end{theorem}
 In the applications, we only use the case $t=1$. However, the proof of this Theorem is by induction on $k$ and requires all $t$. 
\begin{rem}\label{r:g}
	By taking conditional expectations, the starting values $m^{(0)}$, $m^{(1)}$ can be replaced by random ones that are independent of $g$ (the constants in our results do not depend on $m^{(0)}$, $m^{(1)}$). As the proof shows, Theorem~\ref{derivcontrolthrm} also holds if the starting values $m^{(0)}$, $m^{(1)}$ of the iteration depend on $g$ as long as they satisfy \eqref{claim}. The assumption that $m^{(0)}$ and $m^{(1)}$ are independent of $g$ becomes crucial in the proof of~\eqref{e:qprop-chi} in Theorem~\ref{q-prop}.
\end{rem}

Building on Theorems~\ref{deltathrm} and~\ref{derivcontrolthrm}, we now infer:
\begin{proof}[Proof of Theorem \ref{Stein-TAP} (using Theorems~\ref{deltathrm} and~\ref{derivcontrolthrm})]
		It remains to show that the r.h.s. of \eqref{e:Stein-tr}  is bounded in absolute value by $C/\sqrt{N}$, which we  do for each summand, i.e. each $l$.
	Using Gaussian integration by parts, the representation~\eqref{e:gl} of $v_l$, and the definitions~\eqref{e:def-DeltaE}, we obtain
	\begin{equation} \label{steincomp1}  \E\left[Y_{i_l}^{(k_l)} v_l \right]  =  \E\left[ \mathcal{E}_{i_l}^{(k_l)}  v_l \right] +   \E\left[\sum_{j\leq L}\frac{\beta}{\sqrt{N}} \sum_{a\neq i_l}  m^{(k_l)}_a  (\md_{i_l a} Y^{(k_j)}_{i_j})\partial_{e_j} v_l \right]
		+ \E\left[\sum_{j \leq L}\frac{\beta}{\sqrt{N}} \sum_{a\neq i_l}  m^{(k_l)}_a 	\mathcal{F}_{i_l a}^{j}   \right]
	\end{equation}
	where
	\begin{equation}
		\label{e:F}
		\mathcal F_{i_l a}^{j} = -\int_0^\infty e^{-u } \sqrt{1-e^{-2u}} \EE_Z\left[ ((\md_{i_l a} \Sigma^{1/2}) Z)_j \partial_{e_j+e_l} f( \by e^{-u} + \sqrt{1-e^{-2u}} \Sigma^{1/2}Z) \right]du
	\end{equation}
	and $\by =(Y_{i_l}^{(k_j)})_{j\leq L}$.
	
	By simply using that $\left(\E X\right)^2\leq \E\left(X^2\right)$, we get that the last expectation in~\eqref{steincomp1} has absolute value bounded by the square root of
	\begin{multline*} 
	\frac{\beta^2}{N}\sum_{a,a'\neq i_l} \E\left[\left( \sum_{j\leq L}  m^{(k_l)}_a	\mathcal{F}_{i_l a}^{j} \right) \left( \sum_{j'\leq L}  m^{(k_l)}_{a'}	\mathcal{F}_{i_l {a'}}^{j'} \right)  \right] \\ =\frac{\beta^2}{N}\sum_{a,a'\neq i_l} \E\left[m^{(k_l)}_a m^{(k_l)}_{a'}  \left( \sum_{j\leq L}  	\mathcal{F}_{i_l a}^{j} \right) \left( \sum_{j'\leq L}  	\mathcal{F}_{i_l {a'}}^{j'} \right)  \right], \end{multline*}
	which in turn is bounded by
	\[
		\le \frac{\beta^2}{N}\sum_{a,a'\neq i_l} \E\left[\left( \sum_{j\leq L}  	\mathcal{F}_{i_l a}^{j} \right)^2\right]^{1/2} \E\left[\left( \sum_{j'\leq L}  	\mathcal{F}_{i_l {a'}}^{j'} \right)^2  \right]^{1/2} \le C_{\beta,k,L} N^{-1} \]
	by Cauchy-Schwarz, and the following Lemma, which is shown at the end of this proof.
\begin{lem}\label{l:g-deriv}
	Let $\mathcal F_{i_l a }^{j}$ be defined by~\eqref{e:F}. Then, there exists a constant $C=C_{k,\beta,L}$ such that
	\[\E\left[\left(\sum_{j=1}^L\mathcal F_{i_l, a}^{j}\right)^2\right]\le C N^{-2}.\]
\end{lem}
We thus conclude that
	\begin{equation} \label{e:EF} 
	\left| \E\left[\sum_{j \leq L}\frac{\beta}{\sqrt{N}} \sum_{a\neq i_l}  m^{(k_l)}_a 	\mathcal{F}_{i_l a}^{j} \right] \right| \leq C_{\beta,k,L} N^{-1/2}\end{equation} 
	
	To compute the second summand in~\eqref{steincomp1}, we consider derivatives of $Y$, which are given by  
	\begin{equation}\label{Yderiv} 
		\md_{yz} Y_{x}^{(k)} = \frac{\beta}{\sqrt{N}}\sum_{b\neq x} \delta_{\{b,x\}\{y,z\}} m_b^{(k)}+ \Delta^{(k)}_{x;y,z}.\end{equation} 
	Hence we have 
	\begin{equation}\bea 
		\label{covtermcomp} \frac{\beta}{\sqrt{N}} \sum_{a\neq i_l}  m^{(k_l)}_a \md_{i_l a} Y_{i_j}^{(k_j)} &= \frac{\beta^2}{N} \sum_{a\neq i_l,b\neq i_j} \delta_{\{b,i_j\}\{i_l,a\}} m^{(k_l)}_a   m_b^{(k_j)}+ \frac{\beta}{\sqrt{N}} \sum_{a\neq i_l} m^{(k_l)}_a\Delta^{(k_j)}_{i_j;i_l,a}  \\
		& = \beta^2 \delta_{i_l i_j} q^{(k_l,k_j)}_N  + \frac{\beta}{\sqrt{N}} \sum_{a\neq i_l} m^{(k_l)}_a\Delta^{(k_j)}_{i_j;i_l,a}+ O(N^{-1}),
		\eea\end{equation}
	where the $O(N^{-1})$ terms is only due to a difference of form $\pm N^{-1} \beta^2 m_{.}^{(.)}m_{.}^{(.)} $. Plugging \eqref{covtermcomp} into \eqref{steincomp1} and this into  the r.h.s.\ of~\eqref{e:Stein-tr} we see that the $\beta^2 \delta_{i_l i_j} q^{(k_l,k_j)}_N$ term in \eqref{covtermcomp} cancels with the first part of each summand on the r.h.s of \eqref{e:Stein-tr}. Hence up error of size at most $C_{\beta,k,L}/\sqrt{N}$ by \eqref{e:EF},  the $l$-th summand on the r.h.s.\ of~\eqref{e:Stein-tr} equals    
	\begin{equation}\label{e:p:sumE}
		\E[ \mathcal{E}_{i_l}^{(k_l)} v_l]+ \sum_{l'=1}^L \E\left[ \frac{\beta}{\sqrt{N}}\sum_{a\neq i_l} m_{a}^{(k_l)} \Delta_{i_{l'};i_l,a}^{(k_{l'})} \partial_{e_{l'}} v_l\right].
	\end{equation}
 Since $|v_l|_\infty \leq 1$ by \eqref{steinest} and using Jensen or Cauchy-Schwarz the first expectation is at most $ C_{\beta,k} N^{-1/2}$ by Theorem~\ref{deltathrm}. Similarly estimating the second expectation using $|m|,|\partial_{e_l'}v_l|_\infty \leq1$ we obtain 
	$$  |\E\left[ \frac{\beta}{\sqrt{N}}\sum_{a\neq i_l} m_{a}^{(k_l)} \Delta_{i_{l'};i_l,a}^{(k_{l'})} \partial_{e_{l'}} v_l\right]|  \leq \beta \sqrt{N} \max_{a\neq i_l} \E[| \Delta_{i_{l'};i_l,a}^{(k_{l'})}|]  \le C_{\beta,k}   \sqrt{N} \max_{a\neq i_l} \sqrt{\E\left[ \left(\Delta_{i_{l'};i_l,a}^{(k_{l'})}\right)^2\right]}.$$
	Application of Theorem~\ref{deltathrm} shows that the r.h.s. in the last display is bounded by $ C_{\beta,k}N^{-1/2}$. Hence, the r.h.s.\ in~\eqref{e:Stein-tr} is bounded by $C_{\beta,k,L}N^{-1/2}$ as asserted. We now come back to the last missing piece, which is the

\begin{proof}[Proof of Lemma \ref{l:g-deriv}]
	By Gaussian integration by parts,
	\begin{multline}\label{e:g-deriv-p1}	\sum_{j=1}^L\mathcal F_{i_l a}^{j} = -\sum_{j=1}^L\int_0^\infty e^{-u } \sqrt{1-e^{-2u}} \EE_Z\left[ ((\md_{i_l a} \Sigma^{1/2}) Z)_j \partial_{e_j+e_l} f( \by e^{-u} + \sqrt{1-e^{-2u}} \Sigma^{1/2}Z) \right]du\\
		=-\int_0^\infty e^{-u } (1-e^{-2u}) \\
		.\sum_{j,p,t=1}^L \EE_Z\left[ ((\md_{i_l a} \Sigma^{1/2}))_{jp} (\Sigma^{1/2})_{tp} \partial_{e_j+e_l+e_t} f( \by e^{-u} + \sqrt{1-e^{-2u}} \Sigma^{1/2}Z) \right]du.
	\end{multline}
	As $\Sigma^{1/2}$ is symmetric we have $\md_{i_l a} \Sigma=\md_{i_l a} (\Sigma^{1/2}  \Sigma^{1/2}) = \Sigma^{1/2}\md_{i_l a}\Sigma^{1/2} + (\md_{i_l a}\Sigma^{1/2})\Sigma^{1/2}$, the sum in the last line equals
	\begin{align*}
		S(u):&= \EE_Z\left[ {\rm tr}\left((\md_{i_l a} \Sigma^{1/2}) \Sigma^{1/2} (\nabla^2 \partial_{e_l} f)( \by e^{-u} + \sqrt{1-e^{-2u}} \Sigma^{1/2}Z) \right)\right]\\
		&= \EE_Z\left[ {\rm tr}\left( \Sigma^{1/2} (\md_{i_l a} \Sigma^{1/2}) (\nabla^2 \partial_{e_l} f)( \by e^{-u} + \sqrt{1-e^{-2u}} \Sigma^{1/2}Z) \right)\right] \\
		&= \EE_Z\left[ {\rm tr}\left(\tfrac12 (\md_{i_l a} \Sigma)  (\nabla^2 \partial_{e_l} f)( \by e^{-u} + \sqrt{1-e^{-2u}} \Sigma^{1/2}Z) \right)\right].
		\end{align*}
The middle equality due to the trace being invariant under transposition as well as cyclic permutation of products.  
	By Cauchy-Schwarz, the expression in the last display is less than
	\[S(u) \leq  \tfrac12 \EE_Z\left[ \sqrt{{\rm tr}\left( (\md_{i_l a} \Sigma)^2\right)  {\rm tr}\left((\nabla^2 \partial_{e_l} f)^2( \by e^{-u} + \sqrt{1-e^{-2u}} \Sigma^{1/2}Z) \right)}\right].\]
As $|\partial_{e_l+e_r+e_s} f|_\infty\le 1$, it follows that
	\[
	\E\left[S(u)^2\right]\le L^2 \E\left(  {\rm tr}\left( (\md_{i_l a} \Sigma)^2\right)\right)=L^2 \sum_{r=1}^L \sum_{s=1}^L \E\left( (\md_{i_l a} \Sigma_{rs})^2\right).
	\]
	By~\eqref{covarianza},
	\[\left(\md_{i_l a} \Sigma_{rs}\right)^2 \leq \left(\md_{i_l a}  \frac{\beta^2} N \sum_{i=1}^N m^{(k_r)}_i m^{(k_s)}_i\right)^2= \left(\frac{\beta^2} N \sum_{i=1}^N m^{(k_r)}_i \md_{i_l a} m^{(k_s)}_i
	+ \frac{\beta^2} N \sum_{i=1}^N m^{(k_s)}_i \md_{i_l a} m^{(k_r)}_i\right)^2.\]
	Using the fact that $(a+b)^2\leq 2(a^2+b^2)$, we obtain
	\begin{multline*} \E\left[\left( \md_{i_l a} \Sigma_{rs}\right)^2\right] \le 2\frac{\beta^4}{N^2}\sum_{i,j=1}^N \E\left[m_i^{(k_r)}(\md_{i_l a} m^{(k_s)}_i)m_j^{(k_r)}(\md_{i_l a} m^{(k_s)}_j)\right] \\ 
	+\E\left[m_i^{(k_s)}(\md_{i_l a} m^{(k_r)}_i) m_j^{(k_s)}(\md_{i_l a} m^{(k_r)}_j)\right].\end{multline*}
	By Cauchy-Schwarz, the first term on the r.h.s. of the latter is smaller than 
	\[\le 2\frac{\beta^4}{N^2} \sum_{i,j=1}^N \E\left[\left(\md_{i_l a} m^{(k_s)}_i\right)^2\right]^{1/2}
	\E\left[\left(\md_{i_l a} m^{(k_s)}_j\right)^2\right]^{1/2}\le C_{\beta,k} N^{-2},\]
	and exactly the same estimate holds for the second term. The last inequality follows as the second moments in the last display are $\le C_{\beta,k} N^{-1- \I_{i\notin\{i_l,a\}}}$ and $\le C_{\beta,k} N^{-1- \I_{j\notin\{i_l,a'\}}}$ by Theorem~\ref{derivcontrolthrm}. Plugging these estimates into~\eqref{e:g-deriv-p1} and using Cauchy-Schwarz inequality, we obtain
	\begin{multline*}
		\E\left[\left(\sum_{j=1}^L\mathcal F_{i_l, a}^{j}\right)^2\right] = \iint_0^\infty -e^{-u -u'} (1-e^{-2u}) (1-e^{-2u'}) \E\left[S(u)^2\right]^{1/2} \E\left[S(u')^2\right]^{1/2} du\, du'\\
		\le C_{\beta,k,L} N^{-2}
		\end{multline*}
	as asserted.
\end{proof}

\end{proof}

\subsection{Controlling derivatives / Proof of Theorem \ref{derivcontrolthrm}}\label{derivcontrolsec}

 In the proof, we use the following consequence of the H\"older inequality which can be shown by induction: For any $n\in\N$ and random variables $X_1,\ldots,X_n$ in $L_n(\mathbb P)$, we have
 \begin{equation}\label{e:Holder}
 	\E\left[X_1\cdots X_n\right]^{n}\le \E\left[|X_1|^{n}\right]\cdots\E\left[|X_n|^{n}\right].
 \end{equation} 

\begin{proof}[Proof of Theorem \ref{derivcontrolthrm}]  We show \eqref{claim} by induction on $k$. The base clauses, $k=1$ for choice   (\ref{On-new}) and $k=0,1$ for (\ref{On-cla}), hold,  since the starting values are independent of $g$. From here on we may assume that the assertion is correct up to $k-1$, but before this will be helpful some preparation and computation is needed. To this end we write $A\preceq B$ if $ A\leq C B$ for a $C>0$ that depends on $\beta, t, k, |\alpha|$ only and denote by $\tanh^{(s)}$ the $s$-th derivative of $\tanh$. By the generalized Faa di Bruno formula (see e.g.\ \cite[Corollary 2.10]{CS}, \cite{HCF}), for each $\alpha\in\N_0^{\mathfrak p}$ with $|\alpha|\ge 1$,
\begin{equation}\label{FaaDiBruno}
\partial_{\alpha} m^{(k)}_p
	= \alpha !\sum_{(\sigma,s)\in\mathcal P(\alpha)}\tanh^{(|s|)}(h^{(k)}_p) \prod_{\ell} \frac{1}{s_\ell!} \left(\frac{\partial_{\sigma_\ell} h^{(k)}_p }{ \sigma_\ell!}\right)^{s_\ell},
\end{equation}
where 
$$\mathcal P(\alpha):= \left\{ (\sigma,s) \in \bigcup_{u=1}^{|\alpha|} \left((\N_0^{\mathfrak p})^u\times \N^u\right) :\sum_{i} s_i\sigma_i = \alpha, \quad 0< \sigma_1<\ldots< \sigma_u \right\}.$$
Note that "$<$" in this definition can be chosen to be any order  on $\N_0^{\mathfrak p}$ with minimal element $0$. From \eqref{FaaDiBruno} we immediately obtain the representation  
$$		\E\left[\left( \partial_\alpha m^{(k)}_p \right)^{2t}\right]=\sum_{(\sigma^{(b)},s^{(b)})_b\in \mathcal P(\alpha)^{2t}}
		\E\left[\prod_{b=1}^{2t}\alpha !\tanh^{(|s^{(b)}|)}(h^{(k)}) \prod_{\ell} \frac{1}{s^{(b)}_\ell!} \left(\frac{\partial_{\sigma^{(b)}_\ell} h^{(k)}_p }{ \sigma^{(b)}_\ell!}\right)^{s^{(b)}_\ell}\right].$$
Roughly estimating all terms independent of $N$ and using that derivatives of $\tanh$ are bounded (since they are polynomials in $\tanh$) we obtain   

\begin{align*} \preceq& \sum_{(\sigma^{(b)},s^{(b)})_b\in \mathcal P(\alpha)^{2t}}\E\left[\prod_{b=1}^{2t} \prod_{\ell}  \left|\partial_{\sigma^{(b)}_\ell} h^{(k)}_p \right|^{s^{(b)}_\ell}\right] \\ \leq& \sum_{(\sigma^{(b)},s^{(b)})_b\in \mathcal P(\alpha)^{2t}}\prod_{b=1}^{2t} \prod_{\ell} \E\left[ \left(\partial_{\sigma^{(b)}_\ell} h^{(k)}_p \right)^{2t \len(s^{(b)})s^{(b)}_\ell}\right]^{\frac{1}{2t \len(s^{(b)})}}, \end{align*}
the last step by the H\"older inequality~\eqref{e:Holder} and $\len$ mapping a vector to its dimension. Finally we estimate by worst case scenario using that the number of summands does not depend on $N$ to obtain 
\begin{equation}\label{maxest} \preceq  \max_{(\sigma,s)\in \mathcal P(\alpha)}  \prod_{\ell} \E\left[ \left(\partial_{\sigma_\ell} h^{(k)}_p \right)^{2t \len(s) s_\ell}\right]^{\frac{1}{\len(s)}}.\end{equation}
Computing derivatives of $h$ from its definition yields 
\begin{equation} \label{hderiv}\partial_\gamma h_p^{(k)}  = N^{-1/2}\sum_i g_{p,i} \partial_\gamma m_i^{(k-1)}+ N^{-1/2} \sum_i {\I}_{\gamma_{\{p,i\}}>0} \partial_{\gamma-\delta_{\{p,i\}}}m_i^{(k-1)}+ \partial_\gamma \On^{(k-1)}_p\end{equation}
for any $\gamma \in \N_0^{\mathfrak p}$ and writing $\delta_{\{p,i\}}$ for the vector that has all zeros except a single one in the $\{p,i\}$ coordinate. We use $\{$, $\}$ here to indicate that $\{p,i\}$ and $\{i,p\}$ refer to the same coordinate. Also note that by slight abuse of notation if $i=p$ we set the indicator to zero. Combining \eqref{hderiv} with the elementary estimate $(a+b+c)^{2n} \leq 3^{2n} (a^{2n}+b^{2n}+c^{2n})$ we obtain for $m\in \N_0, n\preceq 1$
\begin{multline*}
	\E\left[ \left(\partial_{\gamma} h^{(k)}_p \right)^{2n}\right] \preceq  N^{-n}\E\left[\left(\sum_i g_{p,i} \partial_\gamma m_i^{(k-1)}\right)^{2n}\right] \\+ N^{-n} \E\left[\left(\sum_i {\I}_{\gamma_{\{p,i\}}>0} \partial_{\gamma- \delta_{\{p,i\}}}m_i^{(k-1)}\right)^{2n}\right]+ \E\left[\left(\partial_\gamma \On^{(k-1)}_p\right)^{2n}\right].
\end{multline*}
Each of these expectations will be handled separately by the the following Lemmata: 
\begin{lem}\label{lem2} Assuming that the assertion of Theorem \ref{derivcontrolthrm} holds up to $k-1$ we have for $n,|\gamma|\preceq 1$ that  
$$\E\left[\left(\sum_i g_{p,i} \partial_\gamma m_i^{(k-1)}\right)^{2n}\right] \preceq N^{n (1-|\gamma|-\mathcal{C}_\gamma(p) )}$$
 \end{lem} 
\begin{lem}\label{lem3}Assuming that the assertion of Theorem \ref{derivcontrolthrm} holds up to $k-1$ we have for $n,|\gamma|\preceq 1$ that 
$$  \E\left[\left(\sum_i {\I}_{\gamma_{\{p,i\}}>0} \partial_{\gamma- \delta_{\{p,i\}}}m_i^{(k-1)}\right)^{2n}\right]\preceq N^{n (1-|\gamma|-\mathcal{C}_\gamma(p) ) }$$
 \end{lem} 
\begin{lem}\label{lem4}Assuming that the assertion of Theorem \ref{derivcontrolthrm} holds up to $k-1$ we have for $n,|\gamma|\preceq 1$ that 
$$\E\left[\left(\partial_\gamma \On^{(k-1)}_p\right)^{2n}\right]\preceq N^{n(-|\gamma|-\mathcal{C}_\gamma(p))}$$
 \end{lem} 
Given Lemmata \ref{lem2}, \ref{lem3}, \ref{lem4} we have
$$ \E\left[ \left(\partial_{\gamma} h^{(k)}_p \right)^{2n}\right]  \preceq  N^{-n (|\gamma|+\mathcal{C}_\gamma(p) ) }$$
for $n \preceq 1$. Consequently, by \eqref{maxest} and recalling that
$\sum_{i} s_i\sigma_i = \alpha$,
the claim follows since
$$ \E\left[\left( \partial_\alpha m^{(k)}_p \right)^{2t}\right] \preceq \max_{(\sigma,s)\in \mathcal P(\alpha)}  \prod_{\ell} N^{-t s_\ell ( |\sigma_\ell|+\mathcal{C}_{\sigma_\ell}(p))}=   N^{- t  ( |\alpha|+ \min_{(\sigma,s)\in \mathcal P(\alpha)} \sum_\ell s_\ell\mathcal{C}_{\sigma_\ell}(p))}.$$
This concludes the proof, as $\sum_\ell s_\ell\mathcal{C}_{\sigma_\ell}(p)\geq \mathcal{C}_{\alpha}(p) $ for $(\sigma,s)\in \mathcal P(\alpha)$, where the inequality holds as decomposition into subgraphs can only increase the number of connected components.
\end{proof}
We now return to the proofs for Lemmata  \ref{lem2}, \ref{lem3}, \ref{lem4}.

\begin{proof}[Proof of Lemma \ref{lem2}] 	Since
	\[
	\E\left[ \left( \sum_{i=1}^N g_{p,i} \partial_\gamma m_i^{(k-1)} \right)^{2n}\right]
	=\sum_{i_1,\ldots,i_{2n}=1}^N \E\left[\prod_{l=1}^{2n} g_{p,i_l} \partial_\gamma m^{(k-1)}_{i_l}\right],
	\]
	we define $A_i$ for $i=(i_1,\ldots,i_{2n})$ to be the set or indices with unique $i$'s, i.e the set containing those $l\in\{1,..,2n\}$, s.t. $\#\{l': i_l = i_{l'}\}=1$.
	By Gaussian partial integration with respect to all $g_{p,i_l}$ with $l\in A_i$ we obtain
	\[
	\E\left[\prod_{l=1}^{2n} g_{p,i_l} \partial_\gamma m^{(k-1)}_{i_l}\right]^2 =\E\left[ \left( \prod_{l\in A_i^c} g_{p,i_l}\right)  \left(\prod_{l\in A_i}   {\mathfrak d}_{ p i_l} \right) \prod_{l=1}^{2n } \partial_{\gamma} m^{(k-1)}_{i_{l}}\right]^2,
	\]
	which is by Cauchy-Schwarz at most 
	\begin{equation}\label{e:p-2-CS}
		\leq \E\left[ \left( \prod_{l\in A_i^c} g_{p,i_l}\right)^2\right]  \E\left[\left(\left(\prod_{l\in A_i}  {\mathfrak d}_{ p i_l} \right) \prod_{l=1}^{2n} \partial_{\gamma} m^{(k-1)}_{i_l}\right)^2\right].
	\end{equation}
	By Leibniz rule we have the identity 
	\[
	\left(\prod_{l\in A_i}   {\mathfrak d}_{ p i_l} \right) \prod_{l=1}^{2n} \partial_{\gamma} m^{(k-1)}_{i_l}
	=	\sum_{\tau\in \mathcal{M}(i)} \prod_{l=1}^{2n}\partial_{\gamma +\tau_l } m^{(k-1)}_{i_l},
	\]
	where 
	$$ \mathcal{M}(i):= \left\{ \tau \in(\N_0^{\mathfrak p})^{2n}:   \sum_{l=1}^{2n} \tau_l=\sum_{l\in A_i}\delta_{\{p,i_l\}}\right\}.$$ 
	Since the first term in ~\eqref{e:p-2-CS} is $\preceq 1$ we see that ~\eqref{e:p-2-CS} is bounded by  
	\[ \preceq \sum_{\tau,\tau' \in \mathcal{M}(i)}
		\E\left[\left( \prod_{l=1}^{2n} \partial_{\gamma +\tau_l } m^{(k-1)}_{i_l}\right)\left( \prod_{l'=1}^{2n} \partial_{\gamma +\tau'_{l'} } m^{(k-1)}_{i_l}\right)\right],
	\]
	which in turn, estimating the product by the H\"older inequality~\eqref{e:Holder} and using $|A_i|\preceq 1$, is bounded by 
	\[  \preceq  \left(\max_{\tau\in \mathcal{M}(i)} \prod_{l=1}^{2n} \E\left[\left( \partial_{\gamma +\tau_l } m^{(k-1)}_{i_l} \right)^{4n} \right]^{1/4n} \right)^2.
	\]
	By the assertion of Theorem~\ref{derivcontrolthrm} for $k-1$,
	\[\E\left[\left(  \partial_{\gamma +\tau_l } m^{(k-1)}_{i_l} \right)^{4n} \right]^{1/4n}\preceq \sqrt{N}^{-|\gamma + \tau_l| -\mathcal C_{\gamma + \tau_l}(i_l)}= \sqrt{N}^{-|\gamma|- |\tau_l| -\mathcal C_{\gamma + \tau_l}(i_l)}.
	\]
	Hence we overall obtain the estimate  
	\begin{equation}\label{lem2calc} \E\left[ \left( \sum_{i=1}^N g_{p,i} \partial_\gamma m_i^{(k-1)} \right)^{2n}\right]
	\preceq \sum_{i_1,\ldots,i_{2n}=1}^N   N^{-n|\gamma| - |A_i|/2 -    \min_{\tau\in \mathcal{M}(j)} \sum_{l=1}^{2n}   \mathcal C_{\gamma + \tau_l}(i_l)/2}.\end{equation}

Since $\mathcal C_{\gamma + \tau_l}(i_l)$ and  $\mathcal C_{\gamma + \tau_l}(p)$ are identical up to one component that is possibly not counted, and since all edges in $\tau_l$ contain $p$, we have  
	\begin{equation} \label{graphest}\mathcal C_{\gamma + \tau_l}(i_l) \ge  \mathcal C_{\gamma + \tau_l}(p)- \I_{i_{l}\in\gamma, \tau_l = 0}\ge \mathcal C_\gamma(p)- \I_{i_{l}\in\gamma}- \sum_{l'\in A_i} \I_{i_{l'}\in\gamma, i_{l'}\in\tau_l},\end{equation}
where the second estimate is due to each entry in $\tau_l$ at most removing a single component from the count  by connecting it to $p$.  
	By definition of $A_i$ and $\tau$ note that $\sum_{l=1}^{2n}\I_{i_{l'}\in\tau_l}= 1$ for each $l' \in A_i$. Therefore summing \eqref{graphest} over $l$ yields  
	\begin{equation}\label{graphest2}\sum_{l=1}^{2n}   \mathcal C_{\gamma + \tau_l}(i_l)
	\ge 2n\mathcal C_\gamma(p)- 2\sum_{l\in A_i} \I_{i_{l}\in\gamma}- \sum_{l\in A_i^c} \I_{i_{l}\in\gamma}.\end{equation} 
Now estimating the minimum in \eqref{lem2calc} with \eqref{graphest2} yields 
\[	 \E\left[ \left( \sum_{i=1}^N g_{p,i} \partial_\gamma m_i^{(k-1)} \right)^{2n}\right]
	\preceq  N^{n(1-|\gamma|-\mathcal C_\gamma(p))}\sum_{i_1,\ldots,i_{2n}=1}^N   N^{-n- |A_i|/2 + \sum_{l\in A_i} \I_{i_{l}\in\gamma}+ \sum_{l\in A_i^c} \I_{i_{l}\in\gamma}/2}\]
It remains to show, that the sum on the r.h.s. is $\preceq 1$ for any choice of $\gamma$ with $|\gamma|\preceq 1$. To this end we decompose the sum into two parts. First we sum over all partitions of $\{1,..,2n\}$, $P\in \mathcal{P}_{2n}$ say, and then we decide one by one for each block a number in $\{1,..,N\}$ that was not taken by a previous block. Here we decide first which of the $i$ take the same values, and then we assign these values. This yields  
\begin{equation}\label{e:exp}
	\sum_{i_1,\ldots,i_{2n}=1}^N   N^{-n- |A_i|/2 + \sum_{l\in A_i} \I_{i_{l}\in\gamma}+ \sum_{l\in A_i^c} \I_{i_{l}\in\gamma}/2} =  \sum_{P\in \mathcal{P}_{2n}}\sum_{\substack{1\leq i_{B_1}<\ldots< i_{B_L}\leq N\\ \{B_1,\ldots,B_L\}=P}} \prod_{l=1}^{L}  \sqrt{N}^{ (|B_l|+\I_{|B_l|=1})( \I_{i_{B_l}\in\gamma}-1)}
\end{equation}
since
\begin{multline*} 
-2n- |A_i| + 2\sum_{l\in A_i} \I_{i_{l}\in\gamma}+ \sum_{l\in A_i^c} \I_{i_{l}\in\gamma}
=-2n- |A_i| + \sum_{l\in A_i} \I_{i_{l}\in\gamma}+ \sum_{l=1}^{2n} \I_{i_{l}\in\gamma}
\\=  \sum_{l=1}^L(|B_l|+\I_{|B_l|=1})( \I_{i_{B_l}\in\gamma}-1).\end{multline*}
Estimating by dropping the restriction to distinct values for $i_B$ we obtain 
\[\eqref{e:exp}\leq \sum_{P\in \mathcal{P}_{2n}}\prod_{B\in P}\sum_{i_B=1}^N    \sqrt{N}^{  (|B|+\I_{|B|=1}) (\I_{i_{B}\in\gamma}-1)}. \]
If $i_B \in \gamma$, then the the exponent is zero and there are $|\gamma|$ such choices in the inner sum. If however $i_B \not\in \gamma$, then the exponent is at most $-2$ and therefore the inner sum is at most one. Hence the inner sum is at most $1+|\gamma|$ and using that the number of partitions, factors and $|\gamma|$ are all $\preceq 1$ the claim follows. 
\end{proof}

\begin{proof}[Proof of Lemma \ref{lem3}] Expanding and using the H\"older inequality~\eqref{e:Holder} yields 
\begin{align*}
\E\left[\left(\sum_i {\I}_{\gamma_{\{p,i\}}>0} \partial_{\gamma- \delta_{\{p,i\}}}m_i^{(k-1)}\right)^{2n}\right]&= \sum_{i_1,..,i_{2n}} \E\left[\prod_{l=1}^{2n} {\I}_{\gamma_{\{p,i_l\}}>0}\partial_{\gamma- \delta_{\{p,i_l\}}}m_{i_l}^{(k-1)}\right]\\
&\leq \sum_{i_1,..,i_{2n}}\prod_{l=1}^{2n} {\I}_{\gamma_{\{p,i_l\}}>0} \E\left[ \left(\partial_{\gamma- \delta_{\{p,i_l\}}}m_{i_l}^{(k-1)}\right)^{2n}\right]^{1/2n} \\
&= \left( \sum_i {\I}_{\gamma_{\{p,i\}}>0} \E\left[ \left(\partial_{\gamma- \delta_{\{p,i\}}}m_i^{(k-1)}\right)^{2n}\right]^{1/2n}\right)^{2n}.
\end{align*}
By the assertion of Theorem \ref{derivcontrolthrm} for $k-1$ we estimate 
$$ \preceq  \left( \sum_i {\I}_{\gamma_{\{p,i\}}>0} (\sqrt{N})^{-|\gamma|+1-\mathcal{C}_{\gamma- \delta_{\{p,i\}}}(i)}\right)^{2n}.$$ 
The claim follows, since $|\{i:\gamma_{\{p,i\}}>0\}|\preceq 1$ and $\mathcal{C}_{\gamma- \delta_{\{p,i\}}}(i) \geq \mathcal{C}_{\gamma}(i)=\mathcal{C}_{\gamma}(p)$
if $\gamma_{\{p,i\}}>0$. 
\end{proof}

\begin{proof}[Proof of Lemma \ref{lem4}] 
We start by assuming choice (\ref{On-new}). Then, expanding and applying the H\"older inequality~\eqref{e:Holder} inequality yields 
	\[\E\left[ \left( \partial_\gamma \On^{(k-1)}_p \right)^{2n}\right] \le  \left(\frac{\beta}{\sqrt N}\right)^{2n} \sum_{i_1,\ldots,i_{2n}}\prod_{l=1}^{2n}
	\E\left[   \left(\partial_{\gamma+\delta_{\{i_l,p\}}}m^{(k-1)}_{i_l}\right)^{2n}\right]^{1/2n}.\]
	By the assertion of Theorem~\ref{derivcontrolthrm} for $k-1$ we further estimate 
	\begin{equation}\label{e:p-on-i}
		\preceq \left(\frac{\beta}{\sqrt N}\right)^{2n} \sum_{i_1,\ldots,i_{2n}}\prod_{l=1}^{2n}
	N^{-\tfrac{|\gamma| + 1}{2} -\tfrac{\mathcal C_{\gamma + \delta_{\{i_l,p\}}}(i_l)}{2}}
	= \beta^{2n} N^{-2n - n|\gamma|} \sum_{i_1,\ldots,i_{2n}} N^{-\tfrac12 \sum_{l=1}^{2n}\mathcal C_{\gamma + \delta_{\{i_l,p\}}}(i_l)}.
	\end{equation}
	Using that that 
	\[\mathcal C_{\gamma + \delta_{\{i_l,p\}}}(i_l) = \mathcal C_{\gamma + \delta_{\{i_l,p\}}}(p) \ge \mathcal C_{\gamma}(p) -{\I}_{i_l \in \gamma},\]
	we get
	\begin{equation}\label{e:p-on-i-2}
		\eqref{e:p-on-i} \preceq N^{-2n - n|\gamma|-n \mathcal C_{\gamma}(p) } \sum_{i_1,\ldots,i_{2n}} N^{\frac{1}{2}\sum_{l=1}^{2n}{\I}_{i_l \in \gamma}}=N^{- n|\gamma|-n \mathcal C_{\gamma}(p) } \prod_{l=1}^{2n}\frac{1}{N}\sum_{i_l=1}^{N} \sqrt{N}^{{\I}_{i_l \in \gamma}},
	\end{equation}
	and the claim follows.
Next, we consider choice (\ref{On-cla}).	Using that $(a+b)^{2n}\le 2^{2n} a^{2n} + 2^{2n} b^{2n}$, we obtain
\begin{equation}\label{e:p-On-ng-1}
	\E\left[\left( \partial_{\gamma} \On^{(k-1)}_p \right)^{2n}\right]	\preceq \E\left[\left( \partial_{\gamma} m^{(k-2)}_p \right)^{2n}\right] +N^{-2n}\E\left[\left( \sum_{i=1}^N \partial_{\gamma} \left({m^{(k-1)}_i}^2 m^{(k-2)}_p\right) \right)^{2n}\right].
\end{equation}
By the assertion of Theorem~\ref{derivcontrolthrm} we immediately obtain the desired  $\preceq N^{-n|\gamma| - n\mathcal C_\gamma(p)}$ for the first summand. Analysing the second we observe by the Leibniz rule,
\[\partial_{\gamma} \left({m^{(k-1)}_i}^2 m^{(k-2)}_p\right)
= \sum_{\substack{\rho,\sigma,\tau\in\N_0^{\mathfrak p} \\ \rho+\sigma+\tau = \gamma }}
\binom{\gamma}{\rho\ \sigma\ \tau}
\left(\partial_{\rho} m^{(k-1)}_i \right) \left(\partial_{\sigma} m^{(k-1)}_i \right)
\left(\partial_{\tau} m^{(k-2)}_p \right).
\]
Hence expanding and estimating by the H\"older inequality~\eqref{e:Holder} yields
\begin{multline*}
N^{-2n}\E\left[\left( \sum_{i=1}^N \partial_{\gamma} \left({m^{(k-1)}_i}^2 m^{(k-2)}_p\right) \right)^{2n}\right]\\
\preceq N^{-2n}\sum_{i_1,\ldots,i_{2n}}\sum_{\substack{\rho_l,\sigma_l,\tau_l\in\N_0^{\mathfrak p} \\ \rho_l+\sigma_l+\tau_l = \gamma, l=1,\ldots,2n }}
\prod_{l=1}^{2n}\E\left[
\left(\partial_{\rho_l} m^{(k-1)}_{i_l} \right)^{6n}\right]^{1/6n} \\
 \E\left[\left(\partial_{\sigma_l} m^{(k-1)}_{i_l} \right)^{6n}\right]^{1/6n}\E\left[\left(\partial_{\tau_l} m^{(k-2)}_p \right)^{6n}\right]^{1/6n}.
\end{multline*}
By the assertion of Theorem~\ref{derivcontrolthrm} we further estimate 
\begin{equation} \label{e:p:On-ng}\preceq N^{-2n}\sum_{i_1,\ldots,i_{2n}}\sum_{\substack{\rho_l,\sigma_l,\tau_l\in\N_0^{\mathfrak p} \\ \rho_l+\sigma_l+\tau_l = \gamma, l=1,\ldots,2n }} \prod_{l=1}^{2n} N^{-\tfrac{|\rho_l|}{2} - \tfrac{\mathcal C_{\rho_l}(i_l)}{2}} N^{-\tfrac{|\sigma_l|}{2} - \tfrac{\mathcal C_{\sigma_l}(i_l)}{2}} N^{-\tfrac{|\tau_l|}{2} -\tfrac{\mathcal C_{\tau_l}(p)}{2}}.\end{equation}
By definition of $\mathcal{C}$ 
\[\mathcal C_{\rho_l}(i_l) \ge \mathcal C_{\rho_l}(p) - \I_{i_l \in \gamma},\quad \mathcal C_{\sigma_l}(i_l) \ge \mathcal C_{\sigma_l}(p) - \I_{i_l \in \gamma}\]
as well as 
\[ \mathcal C_{\rho_l}(p) + \mathcal C_{\sigma_l}(p) + \mathcal C_{\tau_l}(p) \ge \mathcal C_{\rho_l + \sigma_l + \tau_l}(p) = \mathcal C_\gamma(p)\]
hold. Using these facts as well as $|\gamma|\preceq 1$ the claim follows:
\[
	\eqref{e:p:On-ng} 	\preceq N^{-n|\gamma| - n\mathcal C_\gamma(p)} \sum_{i_1,\ldots,i_{2n}} N^{\sum_{l=1}^{2n} (\I_{i_l\in\gamma}-1)}=N^{-n|\gamma| - n\mathcal C_\gamma(p)} \prod_{l=1}^{2n}\sum_{i=1}^N  N^{\I_{i_l\in\gamma}-1} \preceq N^{-n|\gamma| - n\mathcal C_\gamma(p)}
\]
\end{proof}

\subsection{Estimating error terms /  Proof of Theorem \ref{deltathrm}}\label{detasec}
Theorem \ref{deltathrm} is a direct consequence of the following Lemmata  by considering $\alpha=0, t=1$ in Lemma \ref{PathwiseLem} and $\alpha=\delta_{\{y,z\}},t=1$ as well as $\alpha=0, t=1$ in Lemma \ref{OnsagerErrorLem}. We write $\mathcal C_{\alpha}$ for the number of connected components in the multigraph associated to $\alpha$. 

\begin{lem}\label{PathwiseLem} For $\beta, h\in \R$, $k,t\in\N$, $N\in\N$, $x, y, z\leq N$ with $y\neq z$, $\alpha\in\N_0^{\mathfrak p}$, there exists $C=C_{k,|\alpha|,\beta,h,t}>0$, s.t.
	\[\E\left[\left(\partial_{\alpha} \left( \Delta^{(k)}_{x;y,z} - \partial_{\delta_{\{y,z\}}} \mathcal E^{(k)}_{x}\right)\right)^{2t}\right]\le C N^{-t\left(|\alpha|+\I_{x\not\in\alpha}+\mathcal C_{\alpha+\delta_{\{y,z\}}}\right)}. \]

\end{lem}

\begin{lem}\label{OnsagerErrorLem} For $\beta\ge 0$, $h\in \R$, $k,t\in\N$, $N\in\N$, $x\le N$, $\alpha\in\N_0^{\mathfrak p}$, there exists $C=C_{k,|\alpha|,\beta,h,t}>0$, s.t.
\[ \E\left[ \left(\partial_\alpha \mathcal{E}_x^{(k)} \right)^{2t}\right]\leq C N^{-t\left(|\alpha| +\I_{x\not\in\alpha}+\mathcal C_\alpha \right)} .\]

\end{lem}

\begin{proof}[Proof of Lemma \ref{PathwiseLem}] Writing $\hat{\alpha}:= \alpha + \delta_{\{y,z\}}$ to shorten notation we have by definition of $\Delta$ and $\mathcal{E}$ that
$$ \partial_{\alpha} \left( \Delta^{(k)}_{x;y,z} - \partial_{\delta_{\{y,z\}}} \mathcal E^{(k)}_{x}\right) = \frac{\beta}{\sqrt N}\sum_j\left(  g_{j,x} \partial_{\hat{\alpha}}m^{(k)}_j-\partial_{\hat{\alpha}+\delta_{\{j,x\}}} m^{(k)}_j  \right)+\frac{\beta}{\sqrt N}\sum_j \I_{\{j,x\}\in \alpha} \partial_{\hat{\alpha}-\delta_{\{j,x\}}}m^{(k)}_j.$$ 
Hence
$$ \E\left[\left(\partial_{\alpha} \left( \Delta^{(k)}_{x;y,z} - \partial_{\delta_{\{y,z\}}} \mathcal E^{(k)}_{x}\right)\right)^{2t}\right]$$
\begin{equation}\label{pathwise lem comp1}\preceq N^{-t}\E\left[\left(\sum_j\left(  g_{j,x} \partial_{\hat{\alpha}}m^{(k)}_j-\partial_{\hat{\alpha}+\delta_{\{j,x\}}} m^{(k)}_j  \right)\right)^{2t}\right]+N^{-t}\E\left[\left(\sum_j \I_{\{j,x\}\in \alpha} \partial_{\hat{\alpha}-\delta_{\{j,x\}}}m^{(k)}_j\right)^{2t}\right].\end{equation}
Using that $|\alpha|,t \preceq 1$ we can estimate the second expectation by  
\begin{equation}\label{pathwise lem comp2}\E\left[\left(\sum_j \I_{\{j,x\}\in \alpha} \partial_{\hat{\alpha}-\delta_{\{j,x\}}}m^{(k)}_j\right)^{2t}\right] \preceq \sum_j \I_{\{j,x\}\in \alpha}\E\left[\left( \partial_{\hat{\alpha}- \delta_{\{j,x\}}}m^{(k)}_j\right)^{2t}\right] \preceq   N^{-t\left(|\alpha|+\mathcal{C}_{\hat{\alpha}}-\I_{x\in \alpha}\right)}\end{equation}
the last step by Theorem \ref{derivcontrolthrm} and since for $x\in\alpha$ we have $\mathcal{C}_{\hat{\alpha}-\delta_{\{j,x\}}}(j) \geq \mathcal{C}_{\hat{\alpha}}(x)\geq \mathcal{C}_{\hat{\alpha}}-\I_{x\in\alpha}$. It remains to control the first expectation in \eqref{pathwise lem comp1}. To this end we expand and collect same $j$, i.e. for $\mathcal{P}_{2t}$ the set of partitions of $\{1,..,2t\}$ and writing $[i]$ for the block in $P$ containing $i$ we have  
\begin{multline*}
\E\left[\left(\sum_j\left(  g_{j,x} \partial_{\hat{\alpha}}m^{(k)}_j-\partial_{\hat{\alpha}+\delta_{\{j,x\}}} m^{(k)}_j  \right)\right)^{2t}\right] =\\
 \sum_{P\in \mathcal{P}_{2t}}\sum_{\substack{1\leq j_{B_1}<\ldots< j_{B_L}\leq N\\ \{B_1,\ldots,B_L\}=P}}\E\left[\prod_{i=1}^{2t}\left(  g_{j_{[i]},x} \partial_{\hat{\alpha}}m^{(k)}_{j_{[i]}}-\partial_{\hat{\alpha}+\delta_{\{j_{[i]},x\}}} m^{(k)}_{j_{[i]}}  \right)\right]
\end{multline*}
By Gaussian partial integratian a factor $g_{j_{[i]},x} \partial_{\gamma}m^{(k)}_{j_{[i]}}-\partial_{\gamma+\delta_{\{j_{[i]},x\}}} m^{(k)}_{j_{[i]}}$ has the same effect as taking the derivative w.r.t. $g_{j_{[i]},x}$ and then multiplying by  $\partial_{\gamma}m^{(k)}_{j_{[i]}}$. I.e. for any random variable $Z$ such that the following is defined we have $\E[ \left( g_{j_{[i]},x} \partial_{\gamma}m^{(k)}_{j_{[i]}}-\partial_{\gamma+\delta_{\{j_{[i]},x\}}} m^{(k)}_{j_{[i]}}\right)Z ] = \E[\partial_{\gamma}m^{(k)}_{j_{[i]}} \partial_{\delta_{\{j_{[i]},x\}}}Z]$. Applying this for  all $i$ such that $|[i]|=1$ and by product rule we obtain  the representation 
$$  =\sum_{P\in \mathcal{P}_{2t}}\sum_{\substack{1\leq j_{B_1}<\ldots< j_{B_L}\leq N\\ \{B_1,\ldots,B_L\}=P}}\sum_{\tau}\E\left[\prod_{\substack{ i=1,..,2t:\\ |[i]|=1}}\partial_{\hat{\alpha}+\tau_i} m^{(k)}_{j_{[i]}}\prod_{\substack{ i=1,..,2t:\\ |[i]|>1}}\left(  g_{j_{[i]},x} \partial_{\hat{\alpha}+\tau_i}m^{(k)}_{j_{[i]}}-\partial_{\hat{\alpha}+\tau_i+\delta_{\{j_{[i]},x\}}} m^{(k)}_{j_{[i]}}  \right)\right],$$
where the $\tau$-sum is over all $\tau_1, \dots, \tau_{2t} \in \N_0^{\mathfrak p}$, s.t. $\sum_i \tau_i = \sum_{i:|[i]|=1} \delta_{\{j_{[i]},x\}}$ and $\{j_{[i]},x\} \not\in \tau_i$. Hence estimating by generalized H\"older inequality yields 
\begin{multline}\label{pathwise lem comp3}  \preceq \sum_{P\in \mathcal{P}_{2t}}\sum_{\substack{1\leq j_{B_1}<\ldots< j_{B_L}\leq N\\ \{B_1,\ldots,B_L\}=P}}\sum_{\tau} \prod_{i: |[i]|=1}\E\left[\left(\partial_{\hat{\alpha}+\tau_i} m^{(k)}_{j_{[i]}}\right)^{2t}\right]^{1/2t} \\
\prod_{i: |[i]|>1}\E\left[\left(  g_{j_{[i]},x} \partial_{\hat{\alpha}+\tau_i}m^{(k)}_{j_{[i]}}-\partial_{\hat{\alpha}+\tau_i+\delta_{\{j_{[i]},x\}}} m^{(k)}_{j_{[i]}}  \right)^{2t}\right]^{1/2t}.\end{multline}
Using that $\E[g_{.,.}^{2t}] \preceq 1$ for any $t\preceq 1$, H\"older inequality and applying Theorem \ref{derivcontrolthrm} gives 
\begin{multline}
\E\left[\left(  g_{j_{[i]},x} \partial_{\hat{\alpha}+\tau_i}m^{(k)}_{j_{[i]}}  \right)^{2t}\right], \E\left[\left(\partial_{\hat{\alpha}+\tau_i+\delta_{\{j_{[i]},x\}}} m^{(k)}_{j_{[i]}}\right)^{2t}\right], \E\left[\left( \partial_{\hat{\alpha}+\tau_i}m^{(k)}_{j_{[i]}}  \right)^{2t}\right]  \\
\preceq N^{-t\left( |\hat{\alpha}|+|\tau_i|+ \mathcal{C}_{\hat{\alpha}+\tau_i}(j_{[i]})\right)}.
\end{multline}
Hence we obtain estimating \eqref{pathwise lem comp3} 
\begin{multline} \E\left[\left(\sum_j\left(  g_{j,x} \partial_{\hat{\alpha}}m^{(k)}_j-\partial_{\hat{\alpha}+\delta_{\{j,x\}}} m^{(k)}_j  \right)\right)^{2t}\right] \\ 
\preceq \sum_{P\in \mathcal{P}_{2t}}\sum_{\substack{1\leq j_{B_1}<\ldots< j_{B_L}\leq N\\ \{B_1,\ldots,B_L\}=P}}\sum_\tau \prod_{i=1}^{2t} N^{- \frac{1}{2}\left( |\hat{\alpha}|+|\tau_i|+ \mathcal{C}_{\hat{\alpha}+\tau_i}(j_{[i]})\right)}. 
\end{multline}
Estimating $\mathcal{C}_{\hat{\alpha}+\tau_i}(j_{[i]}) \geq \mathcal{C}_{\hat{\alpha}} - \I_{j_{[i]}\in \hat{\alpha}}- \sum_{i'} \I_{ j_{[i']}\in \tau_{i}, j_{[i']}\in \hat{\alpha}}$ and collecting terms using $|[i']|=1 \Leftrightarrow j_{[i']}\in \sum_{i} \tau_i$ we obtain  
$$ \leq N^{-t \left(|\alpha|+\mathcal{C}_{\hat{\alpha}}\right)} \sum_{P\in \mathcal{P}_{2t}}\sum_{\substack{1\leq j_{B_1}<\ldots< j_{B_L}\leq N\\ \{B_1,\ldots,B_L\}=P}}\sum_\tau N^{\sum_i \left(\frac{1+\I_{|[i]|=1}}{2}\right)\left(\I_{j_{[i]}\in\hat{\alpha}}-1\right)}.$$
Estimating further by dropping the requirement for distinct $j$ we obtain the result 
\begin{equation}\label{pathwise lem comp4} \preceq N^{-t \left(|\alpha|+\mathcal{C}_{\hat{\alpha}}\right)} \sum_{P\in \mathcal{P}_{2t}} \prod_{b=1}^{L} \sum_{j=1}^{N} N^{\left(\frac{|B_b|+\I_{|B_b|=1}}{2}\right)\left(\I_{j\in\hat{\alpha}}-1\right)}\preceq N^{-t \left(|\alpha|+\mathcal{C}_{\hat{\alpha}}\right)}.\end{equation}
The last step since the sums concerning $P$ and $\tau$ only have $\preceq 1$ many summands and the innermost sum is bounded by $|\hat{\alpha}|+1$ as the exponent of $N$ is at most $-1$ except for the $|\hat{\alpha}|$ many cases where it is zero. Plugging \eqref{pathwise lem comp2} and \eqref{pathwise lem comp4} into \eqref{pathwise lem comp1} yields the claim.\end{proof}

\begin{proof}[Proof of Lemma \ref{OnsagerErrorLem}]
	It suffices to consider choice~\eqref{On-cla} as the assertion is immediate for choice~\eqref{On-new}.
	
	We recall that
	$m^{(k)}_l=\tanh(h+Y^{(k-1)}_l)$.
	By the chain rule and~\eqref{Yderiv}, we have
\beq \bea
{\mathfrak d}_{ xl} m_l^{(k)}
	&= \left(1-{m_l^{(k)}}^2\right)  {\mathfrak d}_{ xl}   Y_l^{(k-1)} \\
	&= \left(1-{m_l^{(k)}}^2\right) \left[\frac{\beta}{\sqrt{N}}\sum_{p\neq l} \delta_{\{l,p\}\{x,l\}} m_p^{(k-1)}+ \Delta^{(k-1)}_{l;x,l}\right]\\
	&= \left(1-{m_l^{(k)}}^2\right) \left[\frac{\beta}{\sqrt{N}} m_x^{(k-1)}+ \Delta^{(k-1)}_{l;x,l}\right].
\eea \eeq

	Hence,
	\[\mathcal E^{(k)}_x=\frac{\beta}{\sqrt{N}}  \sum_l  {\mathfrak d}_{ xl} m_l^{(k)} 
	- \beta^2 \left(1-   q^{(k)}_N   \right) m^{(k-1)}_x
	= \frac{\beta}{\sqrt{N}}\sum_{l} \left(1-{m_l^{(k)}}^2\right) \Delta^{(k-1)}_{l;x,l}.\]
	By the product rule, it follows that
	\begin{align*}
		\partial_\alpha \mathcal{E}_x^{(k)}
	&= \frac{\beta}{\sqrt{N}}\sum_{l} \partial_\alpha \left(1-{m_l^{(k)}}\right) \left(1+{m_l^{(k)}}\right)\Delta^{(k-1)}_{l;x,l}\\
		&= \frac{\beta}{\sqrt{N}}\sum_{l}\sum_{\substack{\gamma,\delta,\varepsilon\in\N_0^{\mathfrak p}:\\\gamma+\delta+\varepsilon=\alpha}}
		\binom{\alpha}{\gamma\, \delta\, \epsilon}
		\left(\I_{|\gamma|=0}-\partial_\gamma {m_l^{(k)}}\right)
		\left(\I_{|\delta|=0}+\partial_\delta {m_l^{(k)}}\right)
		\partial_\varepsilon \Delta^{(k-1)}_{l;x,l}.
	\end{align*}
	Hence, by expansion and the H\"older inequality~\eqref{e:Holder},
	\begin{multline}\label{e:DeltaE-p1}
		\E\left[ \left(\partial_\alpha \mathcal{E}_x^{(k)} \right)^{2t}\right] \\ 
		\preceq N^{-t}\sum_{l_1,\ldots,l_{2t}=1}^N \max_{\substack{\gamma_b,\delta_b,\varepsilon_b\in\N_0^{\mathfrak p}:\\\gamma_b+\delta_b+\varepsilon_b=\alpha}}
		\E\left[ \prod_{b=1}^{2t}
			\left(\I_{|\gamma_b|=0}-\partial_{\gamma_b} {m_{l_b}^{(k)}}\right)
		\left(\I_{|\delta_b|=0}+\partial_{\delta_b} {m_{l_b}^{(k)}}\right)
		\partial_{\varepsilon_b} \Delta^{(k-1)}_{l_b;r,l_b}
		\right]\quad\quad\\
		\le N^{-t}\sum_{l_1,\ldots,l_{2t}=1}^N \max_{\substack{\gamma_b,\delta_b,\varepsilon_b\in\N_0^{\mathfrak p}:\\\gamma_b+\delta_b+\varepsilon_b=\alpha}}
		\prod_{b=1}^{2t} \E\left[ 
		\left(\I_{|\gamma_b|=0}-\partial_{\gamma_b} {m_{l_b}^{(k)}}\right)^{6t}\right]^{1/6t}
		\E\left[\left(\I_{|\delta_b|=0}+\partial_{\delta_b} {m_{l_b}^{(k)}}\right)^{6t}\right]^{1/6t}\\
		\times\E\left[\left(\partial_{\varepsilon_b} \Delta^{(k-1)}_{l_b;x,l_b}\right)^{6t}
		\right]^{1/6t}
	\end{multline}	
	As $(a+b)^{6t}\preceq a^{6t} + b^{6t}$ and by Theorem~\ref{derivcontrolthrm},
	\begin{equation}\label{e:trian-p-DeltaE}
		\E\left[\left(\partial_{\varepsilon_b} \Delta^{(k-1)}_{l_b;x,l_b}\right)^{6t}
		\right]
		\preceq \E\left[\left(\partial_{\varepsilon_b} \partial_{\delta_{\{x,l_b\}}} \mathcal E^{(k-1)}_{l_b}\right)^{6t}
		\right]
		+\E\left[\left(\partial_{\varepsilon_b} \left( \Delta^{(k-1)}_{l_b;x,l_b} - \partial_{\delta_{\{x,l_b\}}} \mathcal E^{(k-1)}_{l_b}\right)\right)^{6t}\right]
	\end{equation}
and
	\[\E\left[ 
	\left(\I_{|\gamma_b|=0}-\partial_{\gamma_b} {m_{l_b}^{(k)}}\right)^{6t}\right]
	\preceq \I_{|\gamma_b|=0} +  \E\left[ 
	\left(\partial_{\gamma_b} {m_{l_b}^{(k)}}\right)^{6t}\right]
	\preceq N^{-3t|\gamma_b| -3t\mathcal C_{\gamma_b}(l_b)}.\]
The second term on the rhs of~\eqref{e:trian-p-DeltaE} is $\preceq N^{-3t\left(|\varepsilon_b| +\mathcal C_{\varepsilon_b+\delta_{\{x,l_b\}}}+ \I_{l_b\not\in \varepsilon_b}\right)}$ by Lemma~\ref{PathwiseLem}.
We conclude by induction. The assertion holds for $k=1$ as $\mathcal{E}^{(1)} \defi 0$, since $m^{(0)}=0$ and $m^{(1)}$ is independent of $g$. From now on, we assume that the assertion holds for $k-1$ for some $k\geq 2$. Then, the first term on the rhs of~\eqref{e:trian-p-DeltaE} is $\preceq N^{-3t\left(|\varepsilon_b + \delta_{\{x,l_b\}}| +\mathcal C_{\varepsilon_b+ \delta_{\{x,l_b\}}}\right)}$ by the induction hypothesis.

Plugging these bounds into~\eqref{e:DeltaE-p1}, and using that
\[ |\gamma_b| + |\delta_b| + |\varepsilon_b| = |\alpha|,\]
\[\mathcal C_{\gamma_b}(l_b)\ge \mathcal C_{\gamma_b} - \I_{l_b\in \gamma_b},\qquad \mathcal C_{\gamma_b} + \mathcal C_{\delta_b} + \mathcal C_{\varepsilon_b+\delta_{\{x,l_b\}}}\ge \mathcal{C}_{\alpha+\delta_{\{x,l_b\}}} \ge \mathcal{C}_{\alpha}+\I_{x\not\in\alpha} -\I_{l_b\in\alpha},\]
as well as $|\alpha|\preceq 1$, we obtain
	\begin{align*}
	&\E\left[ \left(\partial_\alpha \mathcal{E}_x^{(k)} \right)^{2t}\right]\\
	&\preceq 
	N^{-t}\sum_{l_1,\ldots,l_{2t}=1}^N \max_{\substack{\gamma_b,\delta_b,\varepsilon_b\in\N_0^{\mathfrak p}:\\\gamma_b+\delta_b+\varepsilon_b=\alpha}}
	\prod_{b=1}^{2t} N^{-\frac{1}{2}\left(|\gamma_b| +\mathcal C_{\gamma_b}(l_b)+|\delta_b| +\mathcal C_{\delta_b}(l_b)+|\varepsilon_b| + \mathcal C_{\varepsilon_b+ \delta_{\{x,l_b\}}}+ \I_{l_b\not\in\varepsilon_b}\right)}\\
	&\le N^{-t(|\alpha| +\mathcal C_\alpha+ \I_{x\not\in\alpha}) -2 t}  
	\sum_{l_1,\ldots,l_{2t}=1}^N \max_{\substack{\gamma_b,\delta_b,\varepsilon_b\in\N_0^{\mathfrak p}:\\\gamma_b+\delta_b+\varepsilon_b=\alpha}}
	\prod_{b=1}^{2t} N^{\frac{1}{2}\left(\I_{l_b\in\gamma_b} + \I_{l_b\in \delta_b}+\I_{l_b\in\varepsilon_b}+\I_{l_b\in\alpha}\right)}\\
	&= N^{-t(|\alpha| +\mathcal C_\alpha+ \I_{x\not\in\alpha}) - 2t}  
	\left(\sum_{l=1}^N 
	 N^{\I_{l\in\alpha}}\right)^{2t}\le N^{-t\left(|\alpha| +\mathcal C_\alpha+\I_{x\not\in\alpha}  \right)}  
\end{align*}	
	as required.\end{proof}

\section{Covariances of local fields}\label{covlocf-sec}
The aim of this section is to prove the propagation relation from Theorem~\ref{q-prop}.
As first step, we use Theorem~\ref{Stein-TAP} to show the propagation relation in expectation:
\begin{lem}\label{prop-E}
	For $\beta, h\in\R$, $k,k'\ge 2$, there exists $C\defi C_{k,k',\beta, h}>0$ such that
	\[\left| \E\left[q^{(k,k')}_N\right] -  \E\left[\psi\left(q^{(k-1,k'-1)}_N,q^{(k-1)}_N, q^{(k'-1)}_N\right)\right] \right|  \le \frac{C}{\sqrt N}\]
	and
	\[\left| \E\left[q^{(k,1)}_N\right] - \E\left[\chi\left(q^{(k-1)}_N\right) \right]\frac{1}{N}\sum_{i=1}^N  m^{(1)}_i \right|  \le \frac{C}{\sqrt N}.\]
\end{lem}
\begin{proof}
	By the definition~\eqref{def-qkk} of $q^{(k,k')}_N$ and the TAP recursion~\eqref{TAP-rec},
	\[\E\left[q^{(k,k')}_N\right]=\frac{1}{N}\sum_{i=1}^N \E\left[m^{(k)}_i m^{(k')}_i\right]
	=\frac{1}{N}\sum_{i=1}^N  \E\left[ \tanh\left(h+Y^{(k-1)}_i\right)\tanh\left(h+Y^{(k'-1)}_i\right)\right].
	\]
	By Theorem~\ref{Stein-TAP},
	\[=\frac{1}{N}\E\left[\psi(q^{(k-1,k'-1)}_N,q^{(k-1)}_N, q^{(k'-1)}_N)\right] + O(N^{-1/2}).\]
	Analogously,
	\begin{multline*}
		\E\left[q^{(k,1)}_N\right]=\frac{1}{N}\sum_{i=1}^N \E\left[m^{(k)}_i \right] m^{(1)}_i
		=\frac{1}{N}\sum_{i=1}^N\E\left[ \tanh\left(h+Y^{(k-1)}_i\right)\right]    m^{(1)}_i\\
		=\E\left[ \chi\left(q^{(k-1)}_N\right) \right] \frac{1}{N}\sum_{i=1}^N   m^{(1)}_i +O(N^{-1/2}).
	\end{multline*}
\end{proof}
Next we show, using Theorem~\ref{Stein-TAP} again, that $q^{(k,k')}_N$ concentrates around its expectation:
\begin{lem}\label{q-conc}
	For $\beta, h\in\R$, $k,k'\in\N$, there exists $C\defi C_{k,k',\beta, h}>0$ such that 
	\begin{equation*}
		\E\left[\left( q_N^{(k,k')} - \E\left[q_N^{(k,k')}\right]\right)^2\right] \leq \frac{C}{\sqrt N}
	\end{equation*}
	for all $N\in \N$.
\end{lem}
\begin{proof}

	We first consider the case that $k,k'\ge 2$. By the TAP recursion~\eqref{TAP-rec},
	\[
	\E\left[q_N^{(k,k')} \right]
	=\sum_{i=1}^N\frac{1}{N}\E\left[ m^{(k)}_i m^{(k')}_i\right]
	=\sum_{i=1}^N\frac{1}{N}\E\left[\tanh\left(h+Y^{(k-1)}_i\right)\tanh\left( h+ Y^{(k'-1)}_i\right)\right].
	\]
	With $f(x,y)=\tanh(h+x)\tanh(h+y)$ and ${\bf Z}$ defined in Theorem~\ref{Stein-TAP}, we obtain
	\[
	=\sum_{i=1}^N \frac{1}{N}\E\left[\tanh\left(h+ Z_i^{(k-1)}\right)
	\tanh\left(h+ Z_i^{(k'-1)}\right)\right] + O(1/\sqrt{N}).
	\]
	For the second moment, we obtain
	\begin{multline*}
		\E\left[\left(q_N^{(k,k')} \right)^2 \right]
		=\frac{1}{N^2} \sum_{i,j=1}^N
		\E\left[m^{(k)}_i m^{(k')}_i m^{(k)}_j m^{(k')}_j\right]\\
		=\frac{1}{N^2} \sum_{i,j=1}^N\E\left[\tanh\left(h+Y^{(k-1)}_i\right)\tanh\left( h+ Y^{(k'-1)}_i\right)
		\tanh\left(h+Y^{(k-1)}_j\right)\tanh\left( h+ Y^{(k'-1)}_j\right)\right].
	\end{multline*}
	Again by Theorem~\ref{Stein-TAP} with $f(x,y,z,w)=\tanh(h+x)\tanh(h+y)\tanh(h+z)\tanh(h+w)$,
	\begin{multline*}
		=\frac{1}{N^2} \sum_{i,j=1}^N \E\Big[\tanh\left(h+Z_i^{(k-1)}\right)\tanh\left( h+ Z_i^{(k'-1)}\right)\\
		\times\tanh\left(h+Z_j^{(k-1)}\right)\tanh\left( h+ Z_j^{(k'-1)}\right)\Big] + O(1/\sqrt{N}).
	\end{multline*}
	By the independence structure of the covariance matrix $\Sigma$,
	\begin{multline*}
		=\frac{1}{N^2} \sum_{i,j=1}^N \E\left[\tanh\left(h+Z_i^{(k-1,i)}\right)\tanh\left( h+ Z_i^{(k'-1)}\right)\right]\\
		\times\E\left[\tanh\left(h+Z_j^{(k-1)}\right)\tanh\left( h+ Z_j^{(k'-1)}\right)\right]+ O(1/\sqrt{N})\\
		=\E\left[\frac{1}{N} {m^{(k)}}^T m^{(k')}  \right]^2 + O(1/\sqrt N)\qquad\qquad\qquad\qquad\qquad\qquad\qquad\qquad\qquad\qquad.
	\end{multline*}
	Hence,
	\begin{multline*}
		\E\left[\left( \frac{1}{N} {m^{(k)}}^T m^{(k')} - \E\left[\frac{1}{N} {m^{(k)}}^T m^{(k')}\right]\right)^2\right]\\
		=\E\left[\left(\frac{1}{N} {m^{(k)}}^T m^{(k')} \right)^2 \right]
		- \E\left[\frac{1}{N} {m^{(k)}}^T m^{(k')} \right]^2=O(1/\sqrt{N}).
	\end{multline*}
	
	Next we consider the case that $k\ge 2$, $k'=1$. Using that the initial value $m^{(1)}$ is independent of $g$, we obtain
	\[\E\left[ m^{(k)}_i m^{(1)}_i \right] = \E\left[ m^{(k)}_i \right] m^{(1)}_i 
	= \E\left[ \tanh\left(h+Z_i^{(k-1)}\right)\right] m^{(1)}_i + O(1/\sqrt N)\]
	and
	\begin{multline*}
		\E\left[m^{(k)}_i m^{(1)}_i m^{(k)}_j m^{(1)}_j\right]
		= \E\left[m^{(k)}_i  m^{(k)}_j \right] m^{(1)}_i m^{(1)}_j\\
		= \E\left[\tanh\left(h+Z_i^{(k-1)}\right)\tanh\left(h+Z_j^{(k-1)}\right) \right] m^{(1)}_i m^{(1)}_j + O(1/\sqrt N)
	\end{multline*}
	and the above computation for $k,k'\geq 2$ passes through in this case just the same.
	
	The case $k=k'=1$ is trivial as $m^{(1)}$ is assumed independent of $g$.
\end{proof}
We also show concentration of $\psi(q^{(k,k')}_N,q^{(k)}_N, q^{(k')}_N)$ and of $\chi(q^{(k)}_N)$ around the expectation. To this end, we establish the following Lemma:

\begin{lem}\label{bounded}
	For $(X, Y)$ centred multivariate normal distributed and $f,g \in {\mathcal C}^{2}$, bounded and with two bounded derivatives we have   
\[ \frac{d}{d\Cov(X,Y)} \E\left( f\left(X\right) g \left(Y\right) \right) = \E\left( f'\left(X\right) g'\left(Y\right) \right)\]
\[ \frac{d}{d\Var(X)} \E\left( f\left(X\right) g \left(Y\right) \right) =\frac{1}{2} \E\left( f''\left(X\right) g\left(Y\right) \right)\]
\[ \frac{d}{d\Var(Y)} \E\left( f\left(X\right) g \left(Y\right) \right) =\frac{1}{2} \E\left( f\left(X\right) g''\left(Y\right) \right)\].
\end{lem}
\begin{rem}
Since we are interested in $\tanh$ and its derivatives for $f$ and $g$ the boundedness  assumptions will always be satisfied for the purposes of this paper. They allow us to not  worry about existance of expectations, but can be weakened significantly.
\end{rem}
\begin{proof}
We consider the covariance matrix 
	\[\mathcal{T}=\begin{pmatrix} t_2 & t_1 \\ t_1 & t_3\end{pmatrix},\]
with $t_2,t_3>0$ and $|t_1|< \sqrt{t_2 t_3}$, i.e. we restrict ourselfs to the interior of the set of covariance matrices. The general case follows by continuity.   
	For a pair $Z=(Z_1,Z_2)$ of independent standard Gaussians, we construct $(X,Y) := Z^T \mathcal{T}^{1/2}$. Let $\sigma_1$, $\sigma_2$ denote the column vectors of $\mathcal{T}^{1/2}$. We note that restriction to the interior of the set of covariance matrices garantees existance of derivatives of $\mathcal{T}^{1/2}$ with respect to $t_1,t_2,t_3$, which will be needed in a moment. For now we obtain that
	\[\E\left( f\left(X\right) g \left(Y\right) \right)= \E \left[ f\left( \sigma_1^T Z \right)g\left(\sigma_2^T Z\right)\right].\]
Using Gaussian integration by parts and the shorthand $\partial_{t_i}= \frac{d}{dt_i}$ one quickly checks the following preparatory calculations 
\[ \E \left( Z_1  (\partial_{t_i } \sigma_{1,1}) f'\left(X\right)  g \left(Y\right)  \right)=\E \left(  (\partial_{t_i } \sigma_{1,1}) \sigma_{1,1}  f''\left(X\right) g \left(Y\right)  \right)
+  \E \left(  (\partial_{t_i } \sigma_{1,1})\sigma_{2,1}  f'\left(X\right)    g' \left(Y\right)  \right),\]
 \[ \E  \left( Z_2   (\partial_{t_i } \sigma_{1,2}) f'\left(X\right)  g \left(Y\right)  \right)=\E \left( (\partial_{t_i } \sigma_{1,2}) \sigma_{1,2}  f''\left(X\right)  g \left(Y\right)  \right)
+ \E \left( (\partial_{t_i } \sigma_{1,2})\sigma_{2,2}  f'\left(X\right)  g' \left(Y\right)  \right),\]
\[ \E \left( Z_1 (\partial_{t_i }\sigma_{2,1}) f\left(X\right) g' \left(Y\right)  \right)=\E\left( ( \partial_{t_i }\sigma_{2,1})  \sigma_{1,1} f'\left(X\right)   g' \left(Y\right)  \right)\\
 +\E \left(  (\partial_{t_i }\sigma_{2,1} ) \sigma_{2,1} f\left(X\right) g''\left(Y\right)  \right),\]
\[ \E \left( Z_2 (\partial_{t_i }\sigma_{2,2}) f\left(X\right)  g' \left(Y\right)  \right)=\E \left( (\partial_{t_i }\sigma_{2,2}) \sigma_{1,2} f'\left(X\right)  g' \left(Y\right)  \right)\\
+\E \left( (\partial_{t_i }\sigma_{2,2}) \sigma_{2,2}   f\left(X\right)   g'' \left(Y\right)  \right). \]  
We now compute directly 
\begin{align*} \partial_{t_i}\E \left(   f\left(X\right) g \left(Y\right)  \right) =& \partial_{t_i}\EE_{Z}\left[ f\left( \sigma_1^T Z \right)g\left( \sigma_2^T Z\right)\right] \\
=&  \E \left( Z_1   (\partial_{t_i }\sigma_{1,1})    f'\left(X\right) g \left(Y\right)  \right)+ \E \left(  Z_2 (\partial_{t_i } \sigma_{1,2})   f'\left(X\right) g \left(Y\right)  \right)+\\
&+ \E \left( Z_1  (\partial_{t_i }\sigma_{2,1})  f\left(X\right) g' \left(Y\right)  \right)+\E \left( Z_2 (\partial_{t_i }\sigma_{2,2})   f\left(X\right) g' \left(Y\right)  \right)	\end{align*}
and by our preparation
\[=\frac{1}{2} \E \left( \left( \partial_{t_i } \left( \sigma_{1,1}^2+ \sigma_{1,2}^2 \right)   \right) f''\left(X\right) g \left(Y\right)  \right) +\frac{1}{2}  \E \left( \left( \partial_{t_i } \left( \sigma_{2,2}^2+ \sigma_{2,1}^2 \right)   \right) f\left(X\right) g'' \left(Y\right)  \right) + \]
\[+\E\left(  (\partial_{t_i } (\sigma_{1,1}\sigma_{2,1}+\sigma_{1,2}\sigma_{2,2}))  f'\left(X\right)    g' \left(Y\right)  \right),\]
which simply is 
\[=\frac{1}{2}\E \left( \left( \partial_{t_i } t_2    \right) f''\left(X\right) g \left(Y\right)  \right)
+\frac{1}{2}\E \left( \left( \partial_{t_i } t_3   \right) f\left(X\right) g'' \left(Y\right)  \right)+ \E \left(  (\partial_{t_i } t_1)  f'\left(X\right)    g' \left(Y\right)  \right),\]
since $\mathcal{T}^{1/2}$ is symmetric. The claim follows by choosing $i= 1,2,3$. 
\end{proof}

\begin{cor}\label{Lipschitz}
	The functions $\psi:[-1,1]\times [0,1]^2\to[-1,1]$ and $\chi:[0,1]\to[-1,1]$ (and their partial derivatives) are Lipschitz.
\end{cor}

\begin{proof}
It is a direct consequence of Lemma \ref{bounded} since all derivatives of $\tanh$ are bounded.
\end{proof}

\begin{cor}\label{conc-psi}
	For $\beta, h\in\R$, $k,k'\in\N$, there exists $C\defi C_{k,k',\beta, h}>0$ such that 
	\[	\E\left[\left( \psi\left(q^{(k,k')}_N,q^{(k)}_N, q^{(k')}_N\right) - \E\left[\psi\left(q^{(k,k')}_N,q^{(k)}_N, q^{(k')}_N\right)\right] \right)^2\right]\le \frac{C}{\sqrt{N}}\]
	and
	\[	\E\left[\left( \chi\left(q^{(k)}_N\right) - \E\left[\chi\left(q^{(k)}_N\right)\right] \right)^2\right]\le \frac{C}{\sqrt{N}}. \]
\end{cor}
\begin{proof}
	We decompose
	\begin{align*}
		\E&\left[\left( \psi\left(q^{(k,k')}_N,q^{(k)}_N, q^{(k')}_N\right) - \E\left[\psi\left(q^{(k,k')}_N,q^{(k)}_N, q^{(k')}_N\right)\right] \right)^2\right]\\
		\le& \;2\E\left[\left( \psi\left(q^{(k,k')}_N,q^{(k)}_N, q^{(k')}_N\right) - \psi\left(\E\left[q^{(k,k')}_N\right],\E\left[q^{(k)}_N\right], \E\left[q^{(k')}_N\right]\right)\right) ^2\right]+\\
		&+2\E\left[\left(  \E\left[\psi\left(q^{(k,k')}_N,q^{(k)}_N, q^{(k')}_N\right)\right] - \psi\left(\E\left[q^{(k,k')}_N\right],\E\left[q^{(k)}_N\right], \E\left[q^{(k')}_N\right]\right)\right) ^2\right]
	\end{align*}
	using that $(a+b)^2\le 2a^2 + 2b^2$. Here the first term on the rhs is $O(N^{-1/2})$ by Lemmata~\ref{q-conc} and~\ref{Lipschitz}. The second expectation is bounded by
	the first one
	by linearity of the expectation and by Jensen's inequality. This implies the first assertion. The second assertion follows analogously.
\end{proof}
Since all preparations are in place we proceed to the
\begin{proof}[Proof of Theorem~\ref{q-prop}.]
For~\eqref{e:qprop-psi}, we use that
\begin{align*}
\E&\left[\left(q^{(k,k')}_N-  \psi\left(q^{(k-1,k'-1)}_N,q^{(k-1)}_N, q^{(k'-1)}_N\right)\right)^2\right]\\
\le&\; 9\E\left[\left( \psi(q^{(k,k')}_N,q^{(k)}_N, q^{(k')}_N) - \E\left[\psi(q^{(k,k')}_N,q^{(k)}_N, q^{(k')}_N)\right] \right)^2\right]+9\E\left[\left( q_N^{(k,k')} - \E\left[q_N^{(k,k')}\right]\right)^2\right]+\\
&+ 9\left(	\E\left[q^{(k,k')}_N\right] -  \E\left[\psi\left(q^{(k-1,k'-1)}_N,q^{(k-1)}_N, q^{(k'-1)}_N\right)\right] \right)^2
	\end{align*}
where all terms on the r.h.s. are bounded by a constant times $N^{-1/2}$ by Lemmata~\ref{prop-E}, \ref{q-conc} and Corollary~\ref{conc-psi}. Assertion~\eqref{e:qprop-chi} follows analogously.
\end{proof}

\section{Fixed point analysis}\label{fp-sec}

We recall that $q$ denotes the largest fixed point of $\phi$, cfr. \eqref{e:q}, defined by:
\[\phi: \R^+_0 \rightarrow [0,1], \quad x \mapsto \EE \tanh^2\left( h + \be \sqrt{x} Z\right).\]

\begin{theorem}\label{unif_conv_thrm}
	For $h\neq0$ or $h=0,\; |\beta|\leq 1$ 
	$$ \sup_{x\in \R^+_0} \left| \phi^{(k)}(x)-q\right| \rightarrow 0, \mbox{ as } k\rightarrow \infty. $$ 
	For $h=0,\; |\beta|>1$ and $\delta>0$ 
	$$ \sup_{x\geq \delta} \left| \phi^{(k)}(x)-q\right| \rightarrow 0, \mbox{ as } k\rightarrow \infty. $$ 
\end{theorem}
The proof of Theorem \ref{unif_conv_thrm} proceeds in two main steps. We first show that in a bounded number of steps $\phi^{(k)}$ visits $[q/2,2q]$. We then establish that $\phi^{(k)}$ contracts strongly on this interval. The case $h=0, \; |\beta|\leq 1$ needs a slightly different treatment, but turns out to be much simpler. \\  \vspace{.2cm}
\hspace{6cm}
{  \large \bf The way to $[q/2,2q]$} \\
In this part, we show the following Lemma: 
\begin{lem}\label{step1_lem} 
	For $h\neq 0$: $q>0$ is the unique fixed point of $ \phi$ and we have     
	$$\sup_{x\in \R^+_0}\; \inf\{ k\in\N: \phi^{(k)}_\beta(x) \in [q/2,2q]\} <\infty.$$
	For $h=0,|\beta|>1$: The fixed points of $ \phi$ are $0$ and $q>0$; and for any $\delta>0$ 
	$$\sup_{x\geq \delta}\; \inf\{ k\in\N: \phi^{(k)}_\beta(x) \in [q/2,2q]\} < \infty.$$
	For $h=0, |\beta| \leq 1$: $q=0$ is the unique fixed point of $ \phi$ and we have 
	$$ \sup_{x\in \R^+_0} \left| \phi^{(k)}(x)-q\right| \rightarrow 0, \mbox{ as } k\rightarrow \infty. $$ 
\end{lem}
The first step to establish Lemma \ref{step1_lem} is to prove the following monotonicity property. 
\begin{lem}\label{semi_monotonicity_lemma}
	For $h\neq 0$ the fixed point $q$ is unique and strictly positive; and we have 
	$$ \phi(x) > x \mbox{ if } x\in [0,q) \quad \mbox{ and } \quad \phi(x) < x \mbox{ if } x>q $$ 
	For $h=0, |\beta|>1$ the fixed points of $ \phi$ are $0$ and $q>0$; for $h=0, |\beta|\leq 1$ the only fixed point is $q=0$ and in either case we have 
	$$ \phi(x) > x \mbox{ if }x\in(0,q) \quad \mbox{ and } \quad \phi(x) < x \mbox{ if } x>q $$ 
\end{lem}

\begin{proof}
We refer to the function $\phi$ for $\beta=1$ by $\phi_1$. By Proposition A.14.1. of \cite{T2} for $\varphi=\tanh$ we have that $\phi_1(x)/x$ is a strictly decreasing function on $\R^+$. Setting $x=\beta^2 y$ for $\beta\neq 0$ we obtain the same monotonicity for $ \phi(y)/y$. For $h\neq 0$ we have $ \phi(0)>0$, hence $ \phi(x)/x =1$ has a unique solution $q>0$ and $ \phi(x)/x < \phi(q)/q = 1$ if and only if $ x>q$. For $h=0$ we have $ \phi(0)=0$. If $\lim_{x\rightarrow 0^+} \phi(x)/x \leq 1$, then $ \phi(x)/x< 1$ for $x>0$ and $q=0$ is the only fixed point. Otherwise  $\lim_{x\rightarrow 0^+} \phi(x)/x > 1$, then $q>0$ and the claim follows for $x>0$ the same way it did in the $h\neq 0$ case.  Finally we verify the transition for $h=0$ at $\beta=1$: By L'Hospital rule $\lim_{x\rightarrow 0^+} \phi(x)/x = \lim_{x\rightarrow 0^+}\phi'(x)=\beta^2$, since  
	$$ \phi'(x) = \beta \EE_Z\left[\frac{Z}{2\sqrt{x}} (\tanh^2)' (\beta\sqrt{x}Z)\right] = \frac{\beta^2}{2} \EE_Z\left[(\tanh^2)''(\beta\sqrt{x}Z)\right]$$
	and $(\tanh^2)''(0) = 2(1-4\tanh^2(0)+3\tanh^4(0))= 2$.
\end{proof}
We proceed to work out the consequences of such monotonicity. 
\begin{lem}\label{fin_steps_lem}
	Let $I\subset \R$ be a compact interval and $f: I \rightarrow \R$ a continuous function with either $f(x)>x$ for all $x\in I$ or $f(x)<x$ for all $x\in I$, then  
	$$ \sup_{x\in I}\;\inf\{k\in \N: f^{(k)}(x) \not\in I\}<\infty. $$
\end{lem}
\begin{proof}
	By continuity of $f$, and compactness of $I$ we have that $x\mapsto f(x)-x$ attains its minimum on $I$. Using that $f(x)>x$ said minimum is strictly positive. Therefore it takes at most $\text{ \it diam}(I)(\min_{y\in I}\left( f(y)-y)\right)^{-1}<\infty$ many steps for $f^{(k)}(x)$ to leave $I$. The case $f(x)<x$ follows by considering $-I$ and $-x$.    
\end{proof}
We finish preparations for the proof of Lemma \ref{step1_lem} by establishing the following paculiar property of $ \phi$. 

\begin{lem}\label{one_last_trick}
	$$ \phi([q,\infty)) \subset [q/2,1]$$ 
\end{lem}
\begin{proof}
	Clearly since $|\tanh|\leq 1$ we have $ \phi(x)\leq 1$ by construction. Due to symmetries we can assume $h,\beta\geq 0$ without loss of generality. Using symmetry and monotonicity of $\tanh^2$ and the Gaussian density around $0$ 
	$$ \EE[\tanh^2(h+\beta \sqrt{q}Z) \mathbf{1}_{Z>0}] \geq \EE[\tanh^2(h+\beta \sqrt{q}Z) \mathbf{1}_{Z<0}].$$
	Hence for $x\geq q$
	$$ \frac{q}{2} = \frac{1}{2} \phi(q) \leq  \EE[\tanh^2(h+\beta \sqrt{q}Z) \mathbf{1}_{Z>0}]\leq  \EE[\tanh^2(h+\beta \sqrt{x}Z) \mathbf{1}_{Z>0}]\leq \phi(x).$$
\end{proof}

\begin{proof}[Proof of Lemma \ref{step1_lem}] 
	We first tackle the cases where $q>0$. Without loss of generality we assume $x\in [0,1]$, since otherwise we may consider $ \phi^{(k)}(x)$ as $ \phi^{(k-1)}( \phi(x))$, i.e. shifting the iteration index by one and starting in $ \phi(x)\leq 1$ instead of $x$.  By Lemmata \ref{semi_monotonicity_lemma} and \ref{fin_steps_lem} the number of steps needed for $\phi^{(k)}$ to leave $[0,q/2]$ for $h\neq 0$ or $[\delta,q/2]$, for $h=0, |\beta|>1$, $0<\delta<q/2$ is uniformly bounded. After these steps we are either on $[q/2,2q]$, in which case there is nothing to prove, or we are on $(2q,1]$. By Lemmata \ref{semi_monotonicity_lemma} and \ref{fin_steps_lem} the iteration leaves $[2q,1]$ in a uniformly bounded number of steps. Finally applying Lemma \ref{one_last_trick} yields the claim, since the last point in $[2q,1]$ before leaving the interval is mapped to $[q/2,2q)$. Finally for $h=0, \beta \leq 1$ we have $q=0$, hence by Lemma \ref{semi_monotonicity_lemma} the sequence $\phi^{(k)}(x)$ is monotonically decreasing regardless of the starting value. Applying Lemma \ref{fin_steps_lem} we obtain that for any $\delta>0$ the sequence $\phi^{(k)}(x)$ exits $[\delta,1]$ in a bounded number of steps, uniformly over all starting points $x\in [0,1]$. The claim follows.     
\end{proof}

\hspace{3cm}
{  \large \bf Strong contractivity on $[q/2,2q]$} \\

To finish the proof of Theorem \ref{unif_conv_thrm} we use the reduction to $[q/2,2q]$ established in the previous Section and then argue on this interval similarly to the Banach fixed point theorem. To this end we first establish a bound on the derivative of $ \phi$ in Lemma \ref{derivest_lemma}, which together with Lemma \ref{semi_monotonicity_lemma} will give strong estimates for $ \phi$ in Lemma \ref{phi_est_lemma}:  
\begin{lem} \label{derivest_lemma}
	For $x>0$
	$$ \phi'(x) > - \frac{ \phi(x)}{2x}$$  
\end{lem} 
\begin{lem} \label{phi_est_lemma}
	For any $h,\beta$ we have 
	\begin{equation}\label{xgrq_ineq}  q-\frac{1}{2}(x-q) < \phi(x) < x \quad \forall x\in(q,2q] \end{equation}
	\begin{equation}\label{xklq_ineq} x < \phi(x)< q + \frac{1}{3\frac{x}{q}-1}(q-x) \quad \forall x\in[q/2,q) \end{equation}
\end{lem}

\begin{proof}[Proof of Lemma \ref{derivest_lemma}]
	We first tackle the case $\beta=1$ and generalize thereafter. Computing the derivative we have for $x\ge 0$
	\begin{align*} x\phi'_1(x) &=   \int \tanh^2(z)  x\left(-\frac{1}{2x}+ \frac{(z-h)^2}{2x^2}\right) \frac{1}{\sqrt{2\pi x}} \exp\left( - \frac{(z-h)^2}{2x}\right) dz\\
	&=\EE \left[ \tanh^2(h+\sqrt{x}Z) \frac{Z^2 - 1}{2} \right] >-\frac{1}{2}\phi_1(x). \end{align*}
	The last step using that $\EE[Z^2 \tanh^2(h+\sqrt{x}Z) ]>0$.
	The claim now follows for general $\beta$, since 
	$$  \phi'(x) = \beta^2 \phi'_1(\beta^2 x)> - \frac{\phi_1(\beta^2x)}{2x} = -\frac{ \phi(x)}{2x}.$$ 
\end{proof}

\begin{proof}[Proof of Lemma \ref{phi_est_lemma}]
	First note that there is no claim if $q=0$, hence we can assume $q>0$. The estimates comparing $ \phi(x)$ to $x$ follow immediately from Lemma \ref{semi_monotonicity_lemma}. We now prove the remaining two estimates. 
	Clearly for $x\in[q/2,2q]$
	\begin{equation}\label{mvt_eq} \phi(x) - q = \phi(x) - \phi(q) = \int_{q}^x \phi'(z) dz. \end{equation}
	Using Lemma \ref{derivest_lemma} and that $ \phi(z)/z$ is decreasing (see proof of Lemma \ref{semi_monotonicity_lemma}) we have for $x>q$ and $z\in [q,x]$ 
	$$  \phi'(z)> - \frac{ \phi(z)}{2z}>- \frac{ \phi(q)}{2q} = -\frac{1}{2}. $$
	estimating \eqref{mvt_eq} with the above yields \eqref{xgrq_ineq}. By the same argument we obtain for $x\in[q/2,q]$ and $z\in [x,q]$ that 
	$$  \phi'(z)> - \frac{ \phi(z)}{2z}>- \frac{ \phi(x)}{2x}.  $$
	Estimating \eqref{mvt_eq} this way yields 
	$$ \phi(x)< q+ \frac{q-x}{2x} \phi(x),  $$
	which can be rearranged to obtain \eqref{xklq_ineq}
	$$ \phi(x)< q (1- \frac{q-x}{2x})^{-1}= q+\frac{1}{3\frac{x}{q}-1}(q-x). $$
\end{proof}

\begin{proof}[Proof of Theorem \ref{unif_conv_thrm}]
	By Lemma \ref{step1_lem} we may assume $q>0$ (as the claim for $q=0$ is covered by Lemma \ref{step1_lem}) and that we start in $[q/2,2q]$, since otherwise there exists a uniformly bounded number of steps, such that we arrive in this interval. We now iterate one more step and consider 
	$$ a := \min \phi( [q/2,2q]), \quad b:= \max \phi( [q/2,2q]).$$
	Note that $ \phi$ is a continuous function and therefore attains maximum and minimum on a compact. By Lemma \ref{phi_est_lemma} we verify $q/2<a\leq b < 2q$, which is seen by estimating all $x$ by worst case in the given range. With this we established that $\phi^{(k)}(x)$ enters $[a,b]$ in a uniformly bounded number of steps and never leaves thereafter. We now consider a slightly larger interval of the shape $[q-\gamma,q+2\gamma]$ to start and choose $\gamma = \max\{ q-a, \frac{b-q}{2}\}<q/2$, such that $[a,b] \subset [q-\gamma,q+2\gamma]\subset (q/2,2q)$.
	For $x\in [q-\gamma,q)$ we have $3\frac{x}{q}-1\geq 3\frac{a}{q}-1 > \frac{1}{2}$, where we note that $q-\gamma>q/2$. Hence by Lemma \ref{phi_est_lemma} we have 
	\begin{equation}\label{cont_eq_prep} x < \phi(x)< q + 2(1-\varepsilon)(q-x) \quad \forall x\in[q-\gamma,q),\end{equation}
	where $\varepsilon := 1- (6\frac{q-\gamma}{q}-2)^{-1}>0$.
	For $d\in (0,\gamma]$, we have $q-\gamma\leq q-d < q+2d \leq q+2\gamma$. Hence, by \eqref{cont_eq_prep} and  the first inequality of Lemma \ref{phi_est_lemma} we have 
	\begin{equation}\label{contract_eq1} \phi([q-d,q]) \subset \left(q-d, q +2 (1-\varepsilon)d\right),\end{equation}
	\begin{equation}\label{contract_eq2} \phi([q,q+2d]) \subset \left(q-d,q+2d\right).\end{equation}
	We now show that a uniformly bounded amount of steps suffices for $\phi^{(k)}(x)$ to reach any fixed neighborhood of $q$. This implies the claim by considering a sequence of smaller and smaller neighborhoods of form $[q-d,q+2d]$, since the iteration stays in such intervals indefinitely due to \eqref{contract_eq1} and \eqref{contract_eq2}. We now fix an arbitrary $ d \in (0,\gamma]$. The iteration leaves $[q-d,q-(1-\varepsilon) d]$ in a uniformly bounded number of steps by Lemmata \ref{fin_steps_lem} and \ref{semi_monotonicity_lemma}. Hence the iteration exits $[q-d,q-(1-\varepsilon) d]$ to a point on $[q-(1-\varepsilon) d, q+2(1-\varepsilon)d]$ and then never escapes this interval:  to see this, we use \eqref{contract_eq1} and \eqref{contract_eq2}
	with $(1-\varepsilon)d$ in place of $d$.  Finally the iteration also leaves $[q+2(1-\varepsilon)d, q+2d]$ in uniformly boundedly many steps to either $[q-(1-\varepsilon) d, q+2(1-\varepsilon)d]$, in which it stays, or to $[q-d, q-(1-\varepsilon)d]$, from which it enters $[q-(1-\varepsilon) d, q+2(1-\varepsilon)d]$ in a uniformly bounded amount of steps. Hence for any $d$ there exists a finite $K_d=K_{d,\beta,h}\in \N$, such that 
	$$ \phi^{(K_d)}([q-d,q+2d]) \subset [q-(1-\varepsilon)d,q+2(1-\varepsilon)d]. $$
	Repeating this process we see that for any fixed $l\in \N$ the iteration reaches $[q-(1-\varepsilon)^l\gamma,q+2(1-\varepsilon)^l\gamma]$ in a number of steps bounded uniformly over all choices of starting point $x\in\R^{+}_0$ for $h\neq 0$ and $x>\delta$ for $h=0$. The claim follows.
\end{proof}

\section{Proof of Theorem~\ref{thm-contract-cond}}\label{contract-sec}

\begin{proof}[Proof of Theorem~\ref{thm-contract-cond}: The case $k=k'$]
	Let $\epsilon>0$, we have 
	\begin{multline}\label{qconvergence}
		\PP\left(\left|q^{(k)}_N-q\right|> \epsilon\right)\leq   \PP\left(\left|q^{(k)}_N- \phi^{(k-1)}\left(q^{(1)}_N\right)\right|> \frac{\epsilon}{2}\right)+  \PP\left(\left| \phi^{(k-1)}\left(q^{(1)}_N\right)-q\right|> \frac{\epsilon}{2}\right).
	\end{multline}
	To show assertion~\eqref{contract-cond-uni} for $k=k'$, it suffices to prove that the $\lim_{k\to\infty}\limsup_{N\to\infty}$ of the terms on the r.h.s.\ of~\eqref{qconvergence} tends to $0$. 
	By the Markov inequality, the first term on the r.h.s. is up to the factor $4/\varepsilon^2$ at most 
	\begin{align}\begin{split}\label{telescope}
		\E\left[\left(q^{(k)}_N- \phi^{(k-1)}\left(q^{(1)}_N\right)\right)^2\right]&=\E\left[\left(\sum_{i=0}^{k-2} \phi^{(i)}\left(q^{(k-i)}_N \right)- \phi^{(i+1)}\left(q^{(k-i-1)}_N\right)\right)^2\right]\\
		&\leq  k \sum_{i=0}^{k-2}\beta^{4i} \E\left[\left( \phi^{(i)}\left(q^{(k-i)}_N \right)- \phi^{(i+1)}\left(q^{(k-i-1)}_N\right)\right)^2\right]\\    & \leq k \sum_{i=0}^{k-2} \E\left[\left( q^{(k-i)}_N- \phi^{(1)}\left(q^{(k-i-1)}_N\right)\right)^2\right] = O\left(\frac{1}{\sqrt N}\right)
\end{split}
	\end{align}
	the second line by $\left(\sum_{i=1}^{k} a_i\right)^2\leq k \sum_{i=1}^{k} a_i^2$, the third line using the fact that $\phi(t)$ is $\beta^2$-Lipschitz multiple times.
	 For the last estimate in~\eqref{telescope}, we used Theorem \ref{q-prop} recognizing
 \[ \phi\left(q^{(k-i-1)}_N\right)\defi  \psi\left( q^{(k-i-1)}_N, q^{(k-i-1)}_N, q^{(k-i-1)}_N \right).\]	
	If $h \neq 0$ or $h=0$ and $\beta<1$, we bound
	\[\lim_{k\to\infty}\limsup_{N\to\infty}\PP\left(\left| \phi^{(k-1)}\left(q^{(1)}_N\right)-q\right|> \frac{\epsilon}{2}\right)
	\le \lim_{k\to\infty}\PP\left( \sup_{x\in \R^+_0} \left| \phi^{(k)}(x)-q\right|> \frac{\epsilon}{2}\right)=0 \]
	for the second term on the r.h.s.\ of~\eqref{qconvergence}, where the last equality follows from Theorem~\ref{unif_conv_thrm}.
	
	If $h = 0$ and $\beta \geq 1$ and $\liminf_{N\to\infty} q^{(1)}_N >\delta$ for some $\delta>0$ and bound the second term on the r.h.s.\ of~\eqref{qconvergence} by
	\begin{equation}\label{qconvergence2-1}
		\lim_{k\to\infty}\limsup_{N\to\infty}\PP\left(\left| \phi^{(k-1)}\left(q^{(1)}_N\right)-q\right|> \frac{\epsilon}{2}\right)
		\le \lim_{k\to\infty} \PP\left(\sup_{x\ge \delta/2}\left| \phi^{(k-1)}\left(x\right)-q\right|> \frac{\epsilon}{2}\right) =0,
	\end{equation}
	 by Theorem~\ref{unif_conv_thrm}.
	
	If $h = 0$ and $\beta \geq 1$, with $\lim_{N\to\infty} q^{(1)}_N =0$, we replace $q$ with $0$ in the beginning of the proof. Then we bound the second term on the r.h.s.\ of~\eqref{qconvergence} by
	\begin{multline}\label{qconvergence2-2}
		\PP\left( \phi^{(k-1)}\left(q^{(1)}_N\right) > \frac{\epsilon}{2}\right)=\PP\left(\left| \phi^{(k-1)}\left(q^{(1)}_N\right)- \phi^{(k-1)}\left(0\right)\right|> \frac{\epsilon}{2}\right)
		\leq \PP\left(\beta^{2(k-1)} \left|q^{(1)}_N-0\right|> \frac{\epsilon}{2}\right)
	\end{multline}
	using the that $ \phi$ is $\beta^2$-Lipschitz and that $0$ is a fixed point. The r.h.s.\ is tending to $0$ as $N \to \infty$ by assumption.
\end{proof}

To consider the limiting covariances we introduce the convenient notation 
\[ \varphi: [-q,q] \mapsto [-q,q],\; t \to \psi(t, q,q).\].

\begin{lem}\label{l:AT-0}
	Let $h=0$. Then the AT condition~\eqref{e:AT} is equivalent to~$|\beta|\le 1$.
\end{lem}
\begin{proof}
	Clearly, \eqref{e:AT} is satisfied if $\beta\le 1$ as $h=0$. We now assume that $h=0$ and $\beta>1$. Then $q>0$,
	$\varphi(0)=\E\left[ \tanh\left( \beta\sqrt{q}Z'\right)\right]^2=0$, and $\varphi(q)= \phi(q)=q$.
	By the proof of Lemma~2.2 a) of~\cite{b1}, $\varphi'(t)>0, \varphi''(t)\geq0$ for $t \geq 0$. As $q>0$ is a fixed point of $\varphi$, it follows that there exists $x\in(0,q)$ with $\varphi(x)<x$. Hence, $\varphi(q)-\varphi(x)>q-x$ and
	\[\frac{1}{q-x}\int_{x}^q \varphi'(t)\, d t >1.\]
	This in turn implies that $\varphi'(t)>1$ for some $t\in(x,q)$. It follows that $\varphi'(q)>1$ as $\varphi'$ is increasing by Lemma~2.2 a) of~\cite{b1}. By Lemma~2.2 b) of~\cite{b1}, $\varphi'(q)=\beta^2\EE\cosh^{-4}(h+\beta\sqrt{q}Z)$, hence the AT condition~\eqref{e:AT} cannot be satisfied.
\end{proof}

Before we show the remaining part of Theorem~\ref{thm-contract-cond} some final preparation is needed. We show negative covariances increase when propagated, i.e. 
\begin{lem}\label{-t}
For $\beta,h\in \R$ and $t<0$ we have 
\[
\varphi(t) \geq \chi(q)^2 +t,
\]
where $ \chi(q)^2=0$ if and only if h=0.
\end{lem}
Thus $\varphi$ has no fixed points on $[-1,0]$ for $h \neq 0$.
\begin{proof}
 Fix $\beta \in \R ,h \neq 0 \in \R$. There exists only one $q=q_{\beta,h}$  that solves \eqref{e:q}. In this proof, $q$ will be considered as a constant and not a function of $h$ anymore. We recall that $\varphi(t)=\psi(\max(-q,\min(t,q)), q,q)$. We will compare this function with
\begin{equation}
	\Psi(t)=\EE\left[\tanh(\beta Z)\tanh(\beta Z') \right]
\end{equation}
where the expectation is over centered Gaussians $Z',Z''$ such that 
\beq 
\Var(Z)=q,  \quad \Var(Z')=q, \quad \Cov(Z,Z')=\min\{\max\{t, -q\}, q\}.
\eeq
Note that for $-q \leq t \leq q$, we can write $\Psi$ as
\begin{equation}
	\Psi(t)=\EE\left[\tanh\left(\beta \left( \text{sign}(t)\sqrt{|t|}Z+ \sqrt{q-|t|} Z'\right)\right)\tanh \left( \beta \left( \sqrt{|t|}Z+ \sqrt{q-|t|}Z''\right)\right)\right]
\end{equation}
 where the expectation is over independent centered Gaussians $Z, Z',Z''$. Note that $\Psi$ is an odd function. We claim for the moment that for $t\in[-q,0]$ and $h\neq 0$, it holds 
 \begin{equation}\label{the_key1} \varphi'(t) \leq \Psi'(t). \end{equation}
It follows that
\[ \varphi(t)= \varphi(0)- \int_t^0 \varphi'(x) dx \geq \varphi(0)- \int_t^0  \Psi'(x) dx  \]
\[ = \varphi(0) - \Psi(0)+ \Psi(t) = \chi(q)^2 +  \Psi(t),\]
since $\varphi(0) = \chi(q)^2>0$ and $\Psi(0)= 0$. We will now prove that $\Psi(t)\geq t$ for $t \in [-q,0]$. For $t \in [0,q]$, we easily check that $\Psi' \geq 0$ and $\Psi'' \geq 0$ (see \cite{b1}, Lemma 2.2). Thus $\Psi$ is convex, increasing on $ [0,q]$  and we claim that
 \begin{equation}\label{the_key2} 
\Psi(q)=\EE\left[\tanh^2\left(\beta \sqrt{q}Z\right)\right]\leq \EE\left[\tanh^2\left(h+\beta \sqrt{q}Z\right)\right]=q,
 \end{equation}
 (We will see that the proof of this claim is almost the same as the proof of the first claim).
 In addition, $\Psi(0)=0$, thus $\Psi(t)\leq t$ on $ [0,q]$. As $\Psi$ is odd, we have that $\Psi(t)\geq t$ on $ [-q,0]$.
It remains to show \eqref{the_key1} and \eqref{the_key2}. To this end we note that for $X,Y$ independent centered Gaussians with 
\[ \Var(X) = \beta^2 \frac{q+t}{2} \quad \mbox{ and } \quad \Var(Y) =  \beta^2\frac{q-t}{2}\]
we have the representation
\[ \varphi(t) =  \EE[ \tanh ( h + X+Y) \tanh (h+X-Y)] \]
and by computing the derivative and applying Gaussian integration by parts
\[ \varphi'(t) = \beta^2 \EE[ \tanh' ( h + X+Y) \tanh'(h+X-Y)]. \]
We now argue that $\varphi'(t)$ is as a function of $h$ maximized in $h=0$. This is hardly surprising when noticing that $\tanh'= 1-\tanh^2$ is maximized in $0$, as matching the maximum of $\tanh'$ with the maximum of the Gaussian density produces the largest value. To make this intuition precise we define 
\[ g(z) := \EE[ \tanh'(z+Y) \tanh'(z-Y)],\]
so that $\varphi'(t) = \beta^2 \EE[ g(h+X)]$. Since $\tanh'$ is an even function and $Y$ has the same distribution as $-Y$ one quickly checks that $g$ is an even function as well. 
Additionally we verify that $g$ is decreasing on $\R^+$ and increasing otherwise by the following estimate for $0<z'<z$ or $0>z'>z$ using the concavity of $\ln \tanh'$:
\[ g(z) = \EE\left[ \exp \left( \ln \tanh' (Y-z) + \ln \tanh' (Y+z)\right)\right]\]
\[ \leq  \EE\left[ \exp \left( \ln \tanh' (Y-z') + \ln \tanh' (Y+z')\right)\right]  = g(z').\]
Writing $f_X$ for the Gaussian density of $X$ we have 
\[ \varphi'(t) = \int_{-\infty}^{+\infty} g(h+x) f_X(x) dx.  \]
$g$ is even, thus $g'$ is an odd function and we have that 
\beq\bea\label{trick}
\frac{d}{dh} &\int_{-\infty}^{+\infty} g(h+x) f_X(x) dx\\
&= \int g'(h+x) f_X(x) dx = e^{\frac{-h^2}{2 \Var(X)}} \int_{-\infty}^{+\infty} g'(x) f_X(x) e^{\frac{hx}{\Var(X)}} dx\\
&=e^{\frac{-h^2}{2 \Var(X)}} \Bigg(- \int_{+\infty}^{0} g'(-x) f_X(-x) e^{-\frac{hx}{\Var(X)}}  dx+\int_{0}^{+\infty} g'(x) f_X(x) e^{\frac{hx}{\Var(X)}}  dx\Bigg)\\
&=2 e^{\frac{-h^2}{2 \Var(X)}} \int_{0}^{+\infty} g'(x) f_X(x) \sinh\left(\frac{hx}{\Var(X)}\right) dx.
\eea\eeq
using a linear change of variable twice. We proved that $g'$ is negative on $\R^+$ and we have that $ \text{sign$\left(\sinh\left(\frac{hx}{\Var(X)}\right)\right)$}=\text{sign(h)}$ for $x>0$: It follows that 
\[
h \mapsto \int_{-\infty}^{+\infty} g(h+x) f_X(x) dx
\]
 is increasing on $\R^{-}$, decreasing on $\R^{+}$ and thus maximal in $h=0$. We thus proved \eqref{the_key1} and \eqref{the_key2} follows by using the same steps as in \eqref{trick}.\\
We recall that $\chi(q)=\E\left(\tanh\left(h+\beta \sqrt{q} Z\right)\right)=:f(h)$. Note that $f(0)=0$ and  
\[f'(h)=\E\left(1-\tanh^2\left(h+\beta \sqrt{q} Z\right)\right)>0,
\]
which allows us to conclude $\chi(q)^2=f(h)^2>0$ for all $h \neq 0$.
\end{proof}

 We proceed to prove Theorem ~\ref{thm-contract-cond} ($k\neq k'$) in cases. If the AT condition ~\eqref{e:AT} is satisfied $\tilde{q}=q$ and $\varphi(q)=q$ is the unique fixed point of $\varphi$. We thereafter give a proof in the case the AT condition ~\eqref{e:AT} is not satisfied. In this case $\tilde{q}<q$ is a second fixed point of $\varphi$ besides $q$, which is the fundamental reason a more involved argument is required.   

\begin{proof}[Proof of Theorem~\ref{thm-contract-cond}: The case $k \neq k'$ if the AT condition \eqref{e:AT} is satisfied.]
          
	 Under the AT condition~\eqref{e:AT}, Lemma~\ref{l:AT-0} and Theorem~\ref{thm-contract-cond} for $k=k'$ give $q^{(k)}_N\to q$ as $N\to\infty$ followed by $k\to\infty$ in probability. We claim that 
	\begin{equation}\label{conc-phi}
		\lim_{k\to\infty}\sup_{t\in[-q,q]}\left|\varphi^{(k)}(t)- q\right|=0
	\end{equation}
	under the AT condition~\eqref{e:AT}. We first finish the proof assuming claim \eqref{conc-phi} and thereafter return to prove it.
	Note that
	\[ | q^{(k,k')}_N | \le \sqrt{q^{(k)}_N q^{(k')}_N}\]
	by definition of these quantities and by the Cauchy-Schwarz inequality. It follows for each $\delta>0$ that
	\[\lim_{k,k'\to\infty} \limsup_{N\to\infty} \normalcolor \mathbb P\left( | q^{(k,k')}_N |\ge q + \delta \right)=0.\]
	Hence, 
	\[ q^{(k,k')}_N - \psi(q^{(k-1,k'-1)}_N \wedge q, q^{(k-1)}_N, q^{(k'-1)}_N ) \to 0\]
	in probability as $N\to\infty$ followed by $k,k'\to\infty$ by Theorem~\ref{q-prop}. 
	As $\psi$ is Lipschitz (by Lemma~\ref{Lipschitz}) and  as \normalcolor $q^{(k)}_N\to q$ in probability,
	\[\psi(q^{(k-1,k'-1)}_N\wedge q , q^{(k-1)}_N, q^{(k'-1)}_N ) -\psi(q^{(k-1,k'-1)}_N\wedge q, q, q)  \to 0\]
	as $N\to\infty$ followed by $k,k'\to\infty$. 
	The second term in the last display equals $\varphi(q^{(k-1,k'-1)}_N \wedge q)$.
	Hence,
	\[q^{(k,k')}_N -  \varphi(q^{(k-1,k'-1)}_N \wedge q)  \to 0 \]
	as $N\to\infty$ followed by $k,k'\to\infty$. Iterating these steps and using that $\varphi$ is Lipschitz, we obtain
	\begin{equation}\label{qkk-it}
		q^{(k,k')}_N - \varphi^{(\ell)}\left(q^{(k-\ell,k'-\ell)}_N\wedge q\right) \to 0
	\end{equation}
	as $N\to\infty$ followed by $k,k'\to\infty$, for any fixed $l\in \N$. The assertion follows from~\eqref{qkk-it} and \eqref{conc-phi}, by considering large $\ell$.
	
	It remains to show the claim~\eqref{conc-phi}.
	This follows as in Lemma 2.4 b) of~\cite{b1}.  For $h=0$, the AT condition~\eqref{e:AT} and Lemma~\ref{l:AT-0} imply that $\beta\le 1$, and hence that $q=0$ is the unique fixed point of~\eqref{e:q} and $\varphi\equiv 0$. We assume $h\ne 0$ from here on, which implies $q>0$. Let $\epsilon\in(0,q)$. 	From~\eqref{e:q}, we note that $\varphi(q)=q$. By the proof of Lemma~2.2 a) of~\cite{b1}, $\varphi'>0$ and $\varphi''>0$ on $[0,q]$. Under the AT condition~\eqref{e:AT}, we have $0\le \varphi'(q-\epsilon)<\varphi'(q)\le 1$  by Lemma~2.2b of~\cite{b1}. Consequently for $t\in[0,q-\epsilon]$ there exists $\delta=\delta_\epsilon>0$, such that $\varphi'(t)\leq 1-\delta$ and therefore
	\[q\ge \varphi(t) \ge  \varphi(q-\epsilon) - (q-\epsilon-t) \varphi'(q-\epsilon) \ge (q-\epsilon) - \left(q-\epsilon-t\right) (1-\delta) = t + \delta(q-\epsilon-t).\]
Hence we obtain uniform convergence of $\varphi^{(k)}$ to $q$ on $[0,q]$, since for any fixed $\varepsilon>0$ a finite number ($\lceil \delta^{-1} \epsilon^{-1} \rceil$) of steps sufices to ensure that $\varphi^{(k)}$ is within $2\varepsilon$ range of $q$. For $t \in[-q,0]$, Lemma \ref{-t} ensures that $\varphi^{(k)}(t)>0$ after a finite number of steps $k$ and we can conclude the proof if  the AT condition~\eqref{e:AT} is satisfied.
\end{proof}

\begin{proof}[Proof of Theorem~\ref{thm-contract-cond}: The case $k \neq k' $ if the AT condition \eqref{e:AT} is {\bf not} satisfied.] 
For $h=0$,we have $\tilde{q}=0$ and all $q^{(k,1)}_N$, $k>1$ converge to zero by Theorem~\ref{q-prop} \eqref{e:qprop-chi} due to $\chi \defi 0$. This implies convergence to zero of all $q^{(k,k')}$ for $k\neq k'$  by Theorem~\ref{q-prop} \eqref{e:qprop-psi}, since $\psi(0,x,y) =0$ for $h=0$.  Next for arbitrary $h\in \R$ and $\beta=0$ we have $\psi\defi \tanh(h)^2$, $\tilde{q}=q = \tanh^2(h)$ and $m_i^{(k)} = \tanh(h)$ for any $k>1$, which yields the claim. From here on let $h\neq 0$ and $\beta\neq 0$. We define for $q^{(1)}\in[0,1]$, $\bar{m}\in [-\sqrt{q^{(1)}},\sqrt{q^{(1)}}]$, for $k\geq k'\geq 1$
\[ q^{(k+1,k'+1)} := \psi(q^{(k,k')},q^{(k,k)},q^{(k',k')} ), \quad q^{(k+1,1)} := \chi( q^{(k,k)}) \bar{m}.\]
We again shorten $q^{(k,k)}$ to $q^{(k)}$. For the choice $q^{(1)}= q^{(1)}_N$ and $\bar{m}= N^{-1} \sum_i m_i^{(1)}$ by Theorem~\ref{q-prop},  $q^{(k,k')}- q^{(k,k')}_N$ converges to zero in probability  in the large $N$ limit. Note that $q^{(k,k')}$ depends on $N$ only through the $N$ dependent choice of starting condition. It is necessary to make sure that the iteration never fixes to the unstable fixed point $\psi(q,q,q) =q$.  To this end we prove 
\begin{lem}
For $\beta,h \neq 0$ uniformly over all $q^{(1)},\bar{m}$ for each $K>1$ there exists $\delta_K>0$, such that 
\begin{equation}\label{e:notq} \frac{q^{(k,k')}}{\sqrt{q^{(k)}q^{(k')}}}<1-\delta_K  \mbox{ for all } k>k', k' \in\{2,\dots, K\}, q^{(1)}\in [0,1] \mbox{ and } \bar{m} \in [-\sqrt{q^{(1)}},\sqrt{q^{(1)}}].\end{equation}
\end{lem}
\begin{proof}
To obtain the required uniformity over $k$ we consider the following iteration started in $a^{(1)}\in [0,1],b^{(0)} \in [\min \phi,1], |\bar{m}|\leq \sqrt{a^{(1)}}$: For $\ell\geq 1$ let  
\[ a^{(\ell+1)} := \phi(a^{(\ell)}), \quad b^{(\ell)} := \phi(b^{(\ell-1)})\]
\[ c^{(\ell+1)} := \psi( c^{(\ell)}, a^{(\ell)},b^{(\ell)}), \quad c^{(1)} := \chi(b^{(0)})\bar{m}. \]
With this we have for $k+1>k'> 1$, that $q^{(k,k')} = c^{(k')}$ when choosing $a^{(1)}=q^{(1)}$ and $b^{(0)}= q^{(k-k')}$. Note that $k+1>k'$ is required to ensure that $k-k'\geq 2$, which is needed to have $b^{(0)}\geq \min \phi$. The case $k+1=k'$ will be resolved later. 
We now view $a,b,c$ as sequences of continuous functions of the starting values $a^{(1)},b^{(0)},\bar{m}$. Hence it suffices to show for $l>1$ and any starting values that 
\beq\label{claimq} c^{(\ell)}< \sqrt{a^{(\ell)}b^{(\ell)}},\eeq
as compactness of the space of starting values and continuity then implies the claim. Note that $c^{(\ell)}\leq \sqrt{a^{(\ell)}b^{(\ell)}}$ is an immeadiate consequence of Cauchy-Schwarz inequality:
\[ \psi(t,x,y) \leq \sqrt{\psi(x,x,x)\psi(y,y,y)}= \sqrt{ \phi(x) \phi(y)}.\]
For equality to hold $t=x=y$ is required on the set $\{(t,x,y): x,y \in [0,1], t\in [-\sqrt{xy},\sqrt{xy}]\}$, since $ \phi(x)>0$ due to $h\neq 0$. Hence $c^{(\ell+1)} = \sqrt{a^{(\ell+1)}b^{(\ell+1)}}$ can only happen if $c^{(\ell)}= a^{(\ell)}=b^{(\ell)}$. Propagating this downwards we see that $c^{(\ell)} = \sqrt{a^{(\ell)}b^{(\ell)}}$ for some $\ell\geq 1$ implies that $c^{(1)} = a^{(1)} = b^{(1)}$. Now 
\[ |c^{(1)}| = |\chi(b^{(0)})\bar{m}|\leq \sqrt{\phi(b^{(0)}) a^{(1)}}= \sqrt{b^{(1)} a^{(1)}},\]
where equality is only possible if $b^{(0)}=0$, as Jensen is only sharp for almost sure constants. This is a contradiction to $b^{(0)}\geq \min \phi>0$ due to $h\neq 0$ and therefore concludes the prove for $k+1>k\geq 1$. We now consider $q^{(k+1,k)}$ as a funtion of the starting values $q^{(1)}$ and $\bar{m}$. Following the same line of reasoning we have for $k>1$ that $q^{(k+1,k)} = \sqrt{q^{(k+1)}q^{(k)}}$ requires $q^{(k,k-1)} = q^{(k)}= q^{(k-1)}$. This however implies that $q^{(k)}=\phi(q^{(k-1)}) = \phi(q^{(k)})$ and therefore $q^{(k)}=q^{(k-1)}=q$.  Propagating this down again we obtain $q^{(1)}=q$, which contradicts that equality in $|q^{(2,1)}|= |\chi(q^{(1)})\bar{m}|\leq \sqrt{\phi(q^{(1)})q^{(1)}}$ only holds for $q^{(1)}=0$, again due to Jensen beeing only sharp for almost sure constants. 
Hence $ q^{(k+1,k)} < \sqrt{q^{(k+1)}q^{(k)}}$. Again for $k>1$ the claim follows due $\frac{q^{(k+1,k)}}{ \sqrt{q^{(k+1)}q^{(k)}}}$ being a continues function of the starting values, that is defined on a compact, hence a maximum is achieved, which strictly below one.  
\end{proof}

From here on, we only consider the choice $q^{(1)} = q^{(1)}_N$ and $\bar{m} = N^{-1} \sum_i m_i^{(1)}$ only and no longer condider the iteration as a function of starting values. Next we define for $k>k'\geq 1$  
\[ \kappa^{(k+1,k'+1)} := f_{k,k'} (\kappa^{(k,k')}), \mbox{ and }  \kappa^{(k,1)}:= \begin{cases} \frac{\chi(q^{(k-1)}) \bar{m}}{\sqrt{q^{(k)}q^{(1)}}} &,\mbox{ if } q^{(1)}>0 \\ 0 &,\mbox{ else }\end{cases}, \mbox{ where} \]
\[ f_{k,k'}(x):= \frac{\psi( x\sqrt{q^{(k)} q^{(k')}},q^{(k)},q^{(k')}) }{\sqrt{ \phi(q^{(k)}) \phi(q^{(k')})}}.\]
Note that if $q^{(1)}=0$ the value of $\kappa^{(k,1)}$  plays no role in the propagation, since it is multiplied by zero when computing $\kappa^{(k+1,2)}= \psi(0,0,q^{(k)})/\sqrt{\phi(0)\phi(q^{(k)})}$.  By Theorem \ref{q-prop} and remembering $\phi(t)\defi \psi(t,t,t)$ by definition, see \eqref{e:psihatdef} and \eqref{e:psi}, we observe that the correlation functional 
\[ \kappa^{(k,k')}_N  \defi  \frac{ q^{(k,k')}_N}{\sqrt{q^{(k)}_N q^{(k')}_N}}\]
satisfies for any $k>k'> 1$ 
\[ |\kappa^{(k,k')}_N - \kappa^{(k,k')}| \rightarrow 0 \mbox{ as } N\rightarrow \infty\]
in probability.  We remark that $\kappa^{(k,k')}$ may depend on $N$, but only through the starting values $q^{(1)}$ and $\bar{m}$. 
Note that $\psi(t,x,y)$ is strictly increasing in $t\in [-\sqrt{xy}, \sqrt{xy}]$ for any $x,y>0$ (see Lemma \ref{bounded} with $f=g=\tanh$) and therefore so is $f_{k,k'}$ for any $k,k'$. Further both functions are
Lipschitz by corollary \ref{Lipschitz}. Since $q^{(k)} \rightarrow q$  by Theorem~\ref{thm-contract-cond} we obtain, using that the Lipschitz constant is uniformly bounded, uniform convergence of $t\mapsto \psi( t , q^{(k)},q^{(k')})$ to $t\mapsto \psi(t,q,q)$ as well as uniform convergence of $f_{k,k'}$ to $t\mapsto \psi( qt,q,q)/q$. Using that $|\tanh''|\leq 2$ it is simple to check that $\frac{d}{dt} \psi(t,x,y)$ as well as $f'_{k,k'}$ are $4\beta^4$ Lipschitz (see corollary \ref{Lipschitz}). This again gives uniform convergence of $ t\mapsto \frac{d}{dt} \psi(t,q^{(k)},q^{(k')})$ to $t\mapsto \frac{d}{dt} \psi(t,q,q)$ and of $f'_{k,k'}$ to $t\mapsto \frac{d}{dt} \left(\psi(qt,q,q)/q\right)$. Corollary \ref{Lipschitz} gives uniform convergence and bounds for higher derivatives of $f'_{k,k'}$, since all derivatives of $\tanh$ are bounded. By assumption the AT condtion \eqref{e:AT} is not satisfied, which is equivalent to $\frac{d}{dt} \psi(t,q,q)|_{t=q}>1 $, where the derivative is taken from below or equivalently on the set $A:= \{(t,x,y): t\leq \sqrt{xy}, x,y\in [0,1]\}$. Since $\psi'$ is continues on $A$ there exists $\varepsilon>0$ and an in $A$ open neighborhood $U_\varepsilon$ of $(q,q,q)$ such that $\inf \psi'(U_\varepsilon) \geq  1+\varepsilon$. This implies that for sufficiently small $\varepsilon>0$ and sufficiently small $\delta\defi \delta_{\varepsilon}>0$ there exists  $K\in \N$, s.t. $f'_{k,k'}(t)>1+\varepsilon$ if $k,k'\geq K$ and $t> 1-\delta$. This impies for $k,k'\geq K$ and $t\in (1-\delta,1)$ that 
\begin{equation}\label{e:repulsion} f_{k,k'}(t) = f_{k,k'}(1)- \int_t^1 f'_{k,k'}(x) dx \leq 1-(1-t)(1+\epsilon)= t -(1-t)\epsilon .\end{equation}
We now recall that $\tilde{q}$ is the smaller of the two fixed points of the on $[0,1]$ convex and increasing function $t \mapsto \psi(t,q,q)$ and that this function is strictly above the identity on $[-1,\tilde{q})$ and below for $(\tilde{q},q)$ (see Lemma \ref{-t} and  Lemma~2.2 a) of~\cite{b1}). Hence by the established uniform convergence, for any $\delta'>0$,  there exist $\varepsilon'>0$ and $K'\in \N$, s.t. $f_{k,k'}(t)> t+\varepsilon'$ for $t\in [-1,\tilde{q}/q-\delta']$ and $f_{k,k'}(t) < t - \varepsilon'$ for $t \in [\tilde{q}/q+\delta', 1-\delta']$  for any $k,k' \geq K'$. We choose $\delta'=\delta$ for simplicity. With this, we observe for sufficiently large $k$, using additionally that $f_{k,k'}$ is continuous, increasing and convex on $\frac{\tilde{q}}{q}+[-\delta,\delta]$ due to uniform convergence, that $f_{k,k'}$ has exactly one fixed point on $\tilde{q}/q+ [-\delta,\delta]$ as long as $\delta$ is chosen smaller than  $(1-\frac{\tilde{q}}{q})/2$. We will refer to this fixed point by $\alpha_{k,k'}$.  By \eqref{e:notq} $\kappa_{K'+\ell,K'}<1- \delta_{K'}$ for any $\ell>0$. We now see that for large $k,k'$  that $n_{\varepsilon'}= \lceil1/\varepsilon' \rceil$ steps suffice to ensure that  $\kappa_{k+n_{\varepsilon'},k'+n_{\varepsilon'}}>0$. Any positive value $t\in [0, \tilde{q}/q-\delta]$ of $\kappa$ increases by at least $\varepsilon'$ each step and eventually enters $\tilde{q}/q+[-\delta,\delta]$ as it can not jump over $\alpha_{k,k'} \in \frac{\tilde{q}}{q}+[-\delta,\delta]$, since $f_{k,k'}(t)\leq f_{k,k'}( \alpha_{k,k'})=\alpha_{k,k'}$. On the other hand any value $t\in [\tilde{q}/q+\delta, 1-\delta]$ decreases and lands in $[\tilde{q}/q-\delta,\tilde{q}/q+\delta]$ as well due to the same reasons. Finally values on $[1-\delta, 1-\delta_{K}]$ get repulsed from $1$ due to $f_{k,k'}(t) \leq t-(1-t)\varepsilon \leq t- \delta_{K}\varepsilon$ by \eqref{e:repulsion} and land in $\tilde{q}/q+[-\delta,\delta]$ as well. Hence for $k>k'$ both large enough for any $\delta>0$ there exists a number of steps $n_{\delta}$  uniformy bounded over all choices of $q^{(1)},\bar{m}$ such that $\kappa_{k+n_{\delta},k'+n_{\delta}}\in \frac{\tilde{q}}{q}+[-\delta,\delta]$. Since $\tilde{q}>0$ is the smallest fixed point of the on $[0,q]$ convex function $\psi(t,q,q)$ with $\psi(0,q,q)>0$ we have $\frac{d}{dt} \psi(t,q,q)|_{t=\tilde{q}} < 1$. Since $f_{k,k'}$ is increasing and above (below) the identity before (after) $\alpha_{k,k'}$. This implies that $f_{k,k'}(t)$ is always in between $t$ and $\alpha_{k,k'}$. Since $\alpha_{k,k'}\in \tilde{q}/q + [-\delta,\delta]$ we have established that the iteration never leaves  $\tilde{q}/q + [-\delta,\delta]$ once entered. Considering a sequence of $\delta$ approaching zero from above yields convergence of $\kappa_{k,k'}$ to $\tilde{q}/q$ uniformly in the starting values and therefore the claim due to the convergence of $q^{(k)}$ to $q$. 
\end{proof}

\end{document}